\documentclass{amsart}
\usepackage{amsmath}%
\usepackage{amsfonts}%
\usepackage{amssymb}%
\usepackage{graphicx}%
\usepackage{mathtools}%
\usepackage{xcolor}%
\usepackage{nameref,hyperref,url}

\newtheorem{theorem}{Theorem}
\theoremstyle{plain}

\newtheorem{corollary}{Corollary}

\newtheorem{definition}{Definition}

\newtheorem{lemma}{Lemma}

\newtheorem{proposition}{Proposition}
\newtheorem{remark}{Remark}

\numberwithin{equation}{section}
\newcommand{\vertiii}[1]{{\left\vert\kern-0.25ex\left\vert\kern-0.25ex\left\vert #1 
    \right\vert\kern-0.25ex\right\vert\kern-0.25ex\right\vert}}
\newcommand{\vertiiin}[1]{{\vert\kern-0.25ex\vert\kern-0.25ex\vert #1 
    \vert\kern-0.25ex\vert\kern-0.25ex\vert}}

\begin{document}
\title[Spectral stability of the peridynamic fractional Laplacian]{A restricted nonlocal operator bridging together the Laplacian and the fractional Laplacian}
\author[Jos\'e C. Bellido]{Jos\'e C. Bellido}
\address{Jos\'e C. Bellido\hfill\break\indent
E.T.S.I. Industriales \& INEI, Department of Mathematics\hfill\break\indent 
University of Castilla-La Mancha\hfill\break\indent 
Ciudad Real, E-13071 \hfill\break\indent 
Spain \hfill\break\indent } 
\email{josecarlos.bellido@uclm.es}

\author[Alejandro Ortega]{Alejandro Ortega}
\address{Alejandro Ortega\hfill\break\indent
E.T.S.I. Industriales, Department of Mathematics\hfill\break\indent 
University of Castilla-La Mancha\hfill\break\indent 
Ciudad Real, E-13071 \hfill\break\indent 
Spain \hfill\break\indent } 
\email{alejandro.ortega@uclm.es}

\thanks{This work has been supported by the {\it Agencia Estatal de Investigaci\'on, Ministerio de Ciencia e Innovaci\'on} (Spain) through
project MTM2017-83740-P}
\thanks{This paper is in final form and no version of it will be submitted for
publication elsewhere.}
\date{\today}

\begin{abstract}
In this work we introduce volume constraint problems involving the nonlocal operator $(-\Delta)_{\delta}^{s}$, closely related to the fractional Laplacian $(-\Delta)^{s}$, and depending upon a parameter $\delta>0$ called horizon. We study the associated linear and spectral problems and the behavior of these volume constraint problems when $\delta\to0^+$ and $\delta\to+\infty$. Through these limit processes on $(-\Delta)_{\delta}^{s}$ we derive spectral convergence to the local Laplacian and to the fractional Laplacian as $\delta\to 0^+$ and $\delta \to +\infty$ respectively, as well as we prove the convergence of solutions of these problems to solutions of a local Dirichlet problem involving $(-\Delta)$ as $\delta\to0^+$ or to solutions of a nonlocal fractional Dirichlet problem involving $(-\Delta)^s$ as $\delta\to+\infty$.
\end{abstract}
\maketitle


\section{Introduction}
Nonlocal and fractional elliptic problems have attracted a great attention in the mathematical community in the last two decades, coming from fields, among others, as nonlocal diffusion \cite{BuVal,Mazon}, statistical mechanics \cite{AlBe}, continuum mechanics, including peridynamics, \cite{EvBe,KassMenScott,MenDu,Si}, and imaging processing \cite{KinOJo,GilStan,BoElPonScher}. For a good account on nonlocal modeling we refer to \cite{Du}. Nonlocal variational problems are also of importance in the characterization Sobolev spaces \cite{BourBrezMiro,Ponce,GioSpec,Spec,MenSpec,Men}. Interesting surveys, which include an exhaustive list of references, on the fractional Laplacian and nonlocal elliptic problems are \cite{RosOton,MR4043885}.  

In this work we study volume constraint elliptic problems driven by a nonlocal operator closely related to the well-known fractional Laplace operator. In particular, given an open bounded domain $\Omega\subset\mathbb{R}^N$ with Lipschitz boundary and $\delta>0$, a parameter called \textit{horizon}, let us define the problem,
\begin{equation}\label{problema}
     \left\{\begin{array}{rl}
     (-\Delta)_{\delta}^su=f &\quad\mbox{in}\quad \Omega,\\
                         u=0 &\quad\mbox{on}\quad \partial_{\delta}\Omega,\\
		 \end{array}\right.
				\tag{$P_{\delta}^{s}$}
\end{equation}
where, 
\begin{equation}\label{defLaplacian}
(-\Delta)_{\delta}^su(x)=c_{N,s}P.V.\int_{B(x,\delta)}\frac{u(x)-u(y)}{|x-y|^{n+2s}}dy,
\end{equation}
with $c_{N,s}=\frac{2^{2s}s\Gamma(\frac{N}{2}+s)}{\pi^{\frac{N}{2}}\Gamma(1-s)}$ a normalization constant and $\partial_\delta\Omega$ the nonlocal boundary given by 
\begin{equation*}
\partial_{\delta}\Omega=\{y\in\mathbb{R}^N\backslash\Omega: |x-y|<\delta\ \text{for}\ x\in\Omega\}.
\end{equation*}
The operator $(-\Delta)_\delta^s$ is not new, and it has been addressed in different studies in the literature before. In view of the definition of $(-\Delta)_{\delta}^s$, it is clear that long-range interactions are neglected and only those exerted at distance smaller than $\delta>0$ are taken into account, i.e., the horizon $\delta>0$ represents the range of interactions. In this sense, the operator $(-\Delta_\delta^s)$, pertaining to the class of nonlocal elliptic operators, it is clearly inspired by peridynamics, and it could actually be seen as a \textit{peridynamic fractional Laplacian}. Peridynamics is a nonlocal continuum model for Solid Mechanics proposed by Silling in \cite{Si}. The main difference with classical theory relies on the nonlocality, reflected in the fact that points separated by a positive distance exert a force upon each other. Since the use of gradients is avoided, peridynamics is a suitable framework for problems where discontinuities appear naturally, such as fracture, dislocation, or, in general, multiscale materials. Operator $(-\Delta)_\delta^s$ fits into bond-based peridynamics (see \cite{Si}), where the elastic energy is computed through a double integral of a pairwise potential function. In \cite{Siwei} a numerical study comparing $(-\Delta)_\delta^s$ with the fractional Laplacian, the spectral fractional Laplacian and the regional Laplacian is performed. In \cite{Alali}, the Fourier multiplier associated to $(-\Delta)_\delta^s$ is computed and, as a consequence, convergence of $(-\Delta)_\delta^su(x)$ to $(-\Delta)u(x)$, for sufficiently smooth $u$, is obtained as $\delta\to 0^+$ or $s\to 1^-$. Also, $(-\Delta)_\delta^s$ was studied in \cite{DeliaGun} in connection with the fractional Laplacian, $(-\Delta)^s=(-\Delta)_{\infty}^s$, and with the motivation of computing numerical approximations of its solutions. Notice that taking the limit as $\delta\to+\infty$ one recovers, at least formally, the usual nonlocal elliptic problem driven by the fractional Laplace operator with boundary condition on the complementary of the domain $\Omega$. Following the probabilistic interpretation for the fractional Laplacian $(-\Delta)_{\infty}^s$, the operator $(-\Delta)_\delta^s$ can be seen as the infinitesimal generator of a symmetric $2s$-stable L\'evy process restricted to $B(x,\delta)$.

At this point, it is worth mentioning that problem \eqref{problema} and the underlying boundary geometry $\partial_{\delta}\Omega$ play and intermediate role between the classical local problem driven by the Laplace operator $(-\Delta)$, where the boundary condition is imposed on $\partial\Omega$, and the nonlocal problem driven by the standard fractional Laplacian $(-\Delta)_{\infty}^s$, where the boundary condition is imposed on the whole $\mathbb{R}^N\backslash\Omega$.

By means of a change of variables, we can write the singular integral \eqref{defLaplacian} as a
weighted second order differential quotient so that the operator $(-\Delta)_{\delta}^s$ admits the representation
\begin{equation}\label{seconddiff}
(-\Delta)_{\delta}^su(x)=-\frac12c_{N,s}\int_{B(0,\delta)}\frac{u(x+y)-2u(x)+u(x-y)}{|y|^{N+2s}}dy.
\end{equation}
As it happens with the standard fractional Laplacian $(-\Delta)_{\infty}^s$, because of \eqref{seconddiff}, the operator $(-\Delta)_{\delta}^s$ has the following monotonicity property: if $u$ has a global maximum at $x$, then $(-\Delta)_{\delta}^s u(x) \geq 0$ for any horizon $\delta>0$. In addition, if $u$ has a local maximum, then there exists an horizon $\delta=\delta(u)>0$ such that $(-\Delta)_{\delta}^s u(x) \geq 0$.

In this paper the limit properties of $(-\Delta)_\delta^s$, both as $\delta \to 0^+$ and as $\delta \to +\infty$, are addressed. We show convergence of solutions and spectral stability, i.e., convergence of eigenvalues and eigenfunctions, to the local Laplacian and to the fractional Laplacian as $\delta \to 0^+$ and as $\delta \to +\infty$ respectively. Therefore, the operator $(-\Delta)_\delta^s$ is an intermediate operator in between the local Laplacian and the fractional Laplacian. 

This investigation fits into the framework of $\Gamma$-convergence (cf. \cite{Braides}). The results for the case $\delta \to 0^+$ rely on a general $\Gamma$-convergence result from \cite{BellCorPed}, whereas the results for the case $\delta \to+\infty$ are consequence of the $\Gamma$-convergence of the energy associated to $(-\Delta)_\delta^s$ to the one corresponding to $(-\Delta)_\infty^s$ induced by the monotone convergence of the sequence of energies. 

References related to this work are the following. Also in the framework of $\Gamma$-convergence is \cite{BrascoPariniSquassina}, where spectral convergence of the fractional $p\,$-Laplacian to the classical $p\,$-Laplacian as $s\to 1^-$ is shown. Closely related to this work is \cite{AndresMunoz}, where spectral stability for certain nonlocal problems in the case $\delta\to 0^+$ is shown without explicitly appealing to $\Gamma$-convergence. As mentioned before, in \cite{DeliaGun}, the convergence phenomena as $\delta\to+\infty$ is addressed for a class of nonlocal linear problems, including the operator $(-\Delta)_\delta^s$ as a particular case, with a direct approach relying on the linearity of those problems.

\textbf{Organization of the paper}: In Section \ref{functionalsetting} we introduce the appropriate functional setting to deal with problems involving the operator $(-\Delta)_{\delta}^s$ and we present the main results proved in this paper. In Section \ref{preliminaryresults} some results about $\Gamma$-convergence that will be essential to prove the main results of this work are included. In Section \ref{horizon0} we prove the results concerning the behaviour of \eqref{problema} and the eigenvalue problem associated to $(-\Delta)_{\delta}^s$ when one takes the horizon $\delta\to0^+$. Finally, in Section \ref{horizoninfty}  we prove the results concerning the behaviour of \eqref{problema} and the eigenvalue problem associated to $(-\Delta)_{\delta}^s$ when one takes the horizon $\delta\to+\infty$. 
\section{Functional setting and Main Results}\label{functionalsetting}
Let us start recalling the fractional-order Sobolev space $H^s(\Omega)$. Given a regular bounded domain $\Omega\subset\mathbb{R}^N$, let us set 
\begin{equation*}
H^s(\Omega)=\{v\in L^2(\Omega):\|v\|_{H^s(\Omega)}<\infty\},
\end{equation*}
where 
\begin{equation*}
\|v\|_{H^s(\Omega)}^2=\|v\|_{L^2(\Omega)}^2+|v|_{H^s(\Omega)}^2
\end{equation*}
with $|\cdot|_{H^s(\Omega)}$ denoting the Gagliardo semi-norm,
\begin{equation}\label{Gagliseminorm}
|v|_{H^s(\Omega)}^2=\int_{\Omega}\int_{\Omega}\frac{|v(x)-v(y)|^2}{|x-y|^{N+2s}}dydx.
\end{equation}
Also, let $H_0^s(\Omega)$ be the completion of $C_0^{\infty}(\Omega)$ under the norm $\|\cdot\|_{H^s(\Omega)}^2$, i.e.
\begin{equation*}
H_0^s(\Omega)=\overline{\mathcal{C}^{\infty}_0 (\Omega)}^{\| \cdot \|_{H^s(\Omega)}}.
\end{equation*}
Due to \cite[Theorem 11.1]{LionMag}, if $0<s\leq\frac{1}{2}$ then $H_0^s(\Omega)=H^s(\Omega)$ while for $\frac{1}{2}<s<1$ we have the strict inclusion $H_0^s(\Omega)\subsetneq H^s(\Omega)$.\newline 
Next, denoting by $\Omega^c=\mathbb{R}^N\backslash\Omega$, let us set the energy space 
\begin{equation*}
\mathcal{H}_0^s(\Omega)=\{v\in H^s(\mathbb{R}^N): v=0 \text{ on }\Omega^c\}
\end{equation*}
endowed with the norm inherited from $H^s(\mathbb{R}^N)$. Let us note that, given $v\in \mathcal{H}_0^s(\Omega)$, although $v=0$ on $\Omega^c$, the norms $\|v\|_{H^s(\Omega)}$ and $\|v\|_{\mathcal{H}_0^s(\Omega)}$ are not the same. Indeed, denoting by $\mathcal{D}=\big(\mathbb{R}^N\times\mathbb{R}^N\big)\big\backslash\big(\Omega^c\times\Omega^c\big)$, we have the strict inclusion $\Omega\times\Omega\subsetneq\mathcal{D}$. In other words, the norm $\|\cdot\|_{\mathcal{H}_0^s(\Omega)}$ takes into account the interaction between $\Omega$ and $\Omega^c$, i.e.,
\begin{equation*}
\begin{split}
\|v\|_{\mathcal{H}_0^s(\Omega)}^2=&\|v\|_{H^s(\mathbb{R}^N)}^2\\
=&\|v\|_{L^2(\mathbb{R}^N)}^2+\int_{\mathbb{R}^N}\int_{\mathbb{R}^N}\frac{|v(x)-v(y)|^2}{|x-y|^{N+2s}}dydx\\
=&\|v\|_{L^2(\Omega)}^2+\iint\limits_{\mathcal{D}}\frac{|v(x)-v(y)|^2}{|x-y|^{N+2s}}dydx.
\end{split}
\end{equation*}
Therefore, the space $\mathcal{H}_0^s(\Omega)$ is the appropriate space to deal with homogeneous elliptic boundary value problems involving the fractional Laplace operator,
\begin{equation}\label{fracclaplacian}
(-\Delta)_{\infty}^s u(x)=c_{N,s}P.V.\int_{\mathbb{R}^N}\frac{u(x)-u(y)}{|x-y|^{N+2s}}dy.
\end{equation}
On the other hand, by \cite[Theorem 6.5]{DiNezzaPalaValdi}, there exists a constant $S(N,s)>0$ such that
\begin{equation*}
\|v\|_{L^{2_{s}^{*}}(\mathbb{R}^N)}^2\leq S(N,s)\iint\limits_{\mathcal{D}}\frac{|v(x)-v(y)|^2}{|x-y|^{N+2s}}dydx, \quad \forall\,v\in\mathcal{H}_0^s(\Omega),
\end{equation*}
with $2_s^*=\frac{2N}{N-2s}$, the critical fractional Sobolev exponent. Hence, for all $v\in\mathcal{H}_0^s(\Omega)$,
\begin{equation*}
\iint\limits_{\mathcal{D}}\frac{|v(x)-v(y)|^2}{|x-y|^{N+2s}}dydx\leq \|v\|_{\mathcal{H}_0^s(\Omega)}^2\leq C_1(\Omega,N,s)\iint\limits_{\mathcal{D}}\frac{|v(x)-v(y)|^2}{|x-y|^{N+2s}}dydx,
\end{equation*}
with $C_1(\Omega,N,s)=S(N,s)|\Omega|^{\frac{2_s^*-2}{2_s^*}}$. Therefore, we can renormize the space $\mathcal{H}_0^s(\Omega)$ and consider it endowed with the norm
\begin{equation}\label{normHfrac}
\vertiii{v}_{\mathcal{H}_0^s}^2=\iint\limits_{\mathcal{D}}\frac{|v(x)-v(y)|^2}{|x-y|^{N+2s}}dydx.
\end{equation}
Next, given an horizon $\delta>0$, let us define the ({\it nonlocally}) completed domain
\begin{equation*}
\Omega_{\delta}=\Omega\cup\partial_{\delta}\Omega=\{y\in\mathbb{R}^N:\ |x-y|<\delta,\ \text{for}\ x\in\Omega\},
\end{equation*}
and the energy space associated to $(-\Delta)_{\delta}^s$ as
\begin{equation*}
\mathbb{H}^s(\Omega_{\delta})=\{v\in L^2(\Omega_{\delta}): \|v\|_{\mathbb{H}^s(\Omega_{\delta})}<\infty\},
\end{equation*}
where
\begin{equation*}
\|v\|_{\mathbb{H}^s(\Omega_{\delta})}^2=\|v\|_{L^2(\Omega_{\delta})}^2+| v |_{\mathbb{H}^s(\Omega_{\delta})}^2,
\end{equation*}
with
\begin{equation*}
| v |_{\mathbb{H}^s(\Omega_{\delta})}^2=\int_{\Omega_{\delta}}\int_{\Omega_{\delta}\cap B(x,\delta)}\frac{|v(x)-v(y)|^2}{|x-y|^{N+2s}}dydx.
\end{equation*} 
The next result links the semi-norm $|\cdot|_{\mathbb{H}^s(\Omega_{\delta})}$ with the Gagliardo semi-norm \eqref{Gagliseminorm}.
\begin{proposition}\label{belcor}\cite[Proposition 6.1]{BellCor} Let $s\in(0,1)$, $1\leq p<\infty$, $\delta>0$ and $\Omega\subset\mathbb{R}^N$ be a bounded Lipschitz domain. Then, there exists $C=C(\delta)>0$ such that for all $u\in W^{s,p}(\Omega)$,
\begin{equation*}
\int_{\Omega}\int_{\Omega}\frac{|u(x)-u(y)|^p}{|x-y|^{N+sp}}dydx\leq C\int_{\Omega}\int_{\Omega\cap B(x,\delta)}\frac{|u(x)-u(y)|^p}{|x-y|^{N+sp}}dydx.
\end{equation*}
\end{proposition}
Because of Proposition \ref{belcor}, the space $\mathbb{H}^s(\Omega_{\delta})$ is isomorphic to the (classical) fractional-order Sobolev space $H^s(\Omega_\delta)$. In order to deal with the boundary value problem \eqref{problema}, we define the energy space
\begin{equation*}
\mathbb{H}_0^{\delta,s}(\Omega)=\{v\in \mathbb{H}^s(\Omega_{\delta}):v\equiv 0\ \text{on}\ \partial_{\delta}\Omega\},
\end{equation*}
endowed with the norm inherited from $\mathbb{H}^s(\Omega_{\delta})$. Let us notice that, given a function $v\in \mathbb{H}_0^{\delta,s}(\Omega)$, although we have $v=0$ on $\partial_{\delta}\Omega=\Omega_{\delta}\backslash\Omega$, the norms $\|v\|_{\mathbb{H}^s(\Omega)}$ and $\|v\|_{\mathbb{H}_0^{\delta,s}(\Omega)}$ are not the same. Indeed, if $v=0$ on $\Omega^c$, since 
\begin{equation*}
\mathbb{H}^s(\Omega)=\{v\in L^2(\Omega): \|v\|_{\mathbb{H}^s(\Omega)}<\infty\},
\end{equation*}
where
\begin{equation*}
\|v\|_{\mathbb{H}^s(\Omega)}^2=\|v\|_{L^2(\Omega)}^2+| v |_{\mathbb{H}^s(\Omega)}^2,
\end{equation*}
with
\begin{equation*}
| v |_{\mathbb{H}^s(\Omega)}^2=\int_{\Omega}\int_{\Omega\cap B(x,\delta)}\frac{|v(x)-v(y)|^2}{|x-y|^{N+2s}}dydx,
\end{equation*} 
denoting by 
\begin{equation*}
\mathcal{D}_{\delta}=\Big(\Omega_{\delta}\times\big(\Omega_{\delta}\cap B(x,\delta)\big)\Big)\Big\backslash\Big(\partial_{\delta}\Omega\times\big(\partial_{\delta}\Omega\cap B(x,\delta)\big)\Big),
\end{equation*}
we have the strict inclusion $\big(\Omega\times(\Omega\cap B(x,\delta))\big)\subsetneq\mathcal{D}_{\delta}$. Hence, the norm $\|\cdot\|_{\mathbb{H}_0^{\delta,s}(\Omega)}$ takes into account the interaction between $\Omega$ and $\partial_{\delta}\Omega$ in the sense that
\begin{equation*}
\begin{split}
\|v\|_{\mathbb{H}_0^{\delta,s}(\Omega)}^2=&\|v\|_{\mathbb{H}^s(\Omega_{\delta})}^2\\
=&\|v\|_{L^2(\Omega_{\delta})}^2+\int_{\Omega_{\delta}}\int_{\Omega_{\delta}\cap B(x,\delta)}\frac{|v(x)-v(y)|^2}{|x-y|^{N+2s}}dydx\\
=&\|v\|_{L^2(\Omega)}^2+\iint\limits_{\mathcal{D}_{\delta}}\frac{|v(x)-v(y)|^2}{|x-y|^{N+2s}}dydx.
\end{split}
\end{equation*}
Therefore, the space $\mathbb{H}_0^{\delta,s}(\Omega)$ is the appropriate space to deal with homogeneous elliptic boundary value problems involving the operator $(-\Delta)_{\delta}^s$.\newline
Some comments are in order. Comparing the norms $\|\cdot\|_{\mathcal{H}_0^s(\Omega)}$ and $\|\cdot\|_{\mathbb{H}_0^{\delta,s}(\Omega)}$ we observe that $\partial_{\delta}\Omega$ plays the role of $\Omega^c$. Indeed, the sets $\Omega_{\delta}$ and $\Omega_{\delta}\cap B(x,\delta)$ will lead to the complete space $\mathbb{R}^N$ in the limit $\delta\to+\infty$, the set $\Omega\cap B(x,\delta)$ will eventually reach the set $\Omega$ for $\delta>0$ big enough and the sets $\partial_{\delta}\Omega$ and $\partial_{\delta}\Omega\cap B(x,\delta)$ will reach $\Omega^c$ in the limit $\delta\to+\infty$. In fact,
\begin{equation}\label{contenido}
\mathcal{D}_{\delta_1}\subset\mathcal{D}_{\delta_2}\quad\text{ for }\delta_1<\delta_2,
\end{equation}
and\footnote{This convergence can be understood in the sense of $cap(\mathcal{D}\backslash\mathcal{D}_{\delta})\to0$ as $\delta\to+\infty$, being $cap(E)=\inf \{ \|v\|_{H^1(\mathbb{R}^N)}^2 : v \in C_0^\infty(\mathbb{R}^N),\ v \geq 1 \text{ on }E\}$, the capacity of the set $E$.}
\begin{equation*}
\mathcal{D}_{\delta}\to\mathcal{D}\quad\text{ as }\delta\to+\infty.
\end{equation*} 
To continue, we recall now a Poincar\'e-type inequality.
\begin{lemma}\cite[Lemma 6.2]{BellCor}\label{poincare}
Let $\Omega\subset\mathbb{R}^N$ a bounded Lipschitz domain and let $\Omega_{D}$ a measurable set of $\Omega$ of positive measure. Let $s\in(0,1)$ and $1\leq p<\infty$. Then, there exists $C>0$ such that for all $u\in W^{s,p}(\Omega)$ with $u=0$ a.e. on $\Omega_{D}$ we have
\begin{equation*}
\|u\|_{L^p(\Omega)}\leq C\ |u|_{W^{s,p}(\Omega)}=C\left(\int_{\Omega}\int_{\Omega}\frac{|u(x)-u(y)|^p}{|x-y|^{N+sp}}dydx\right)^{\frac{1}{p}}.
\end{equation*}
\end{lemma}
Because of Proposition \ref{belcor} and Lemma \ref{poincare}, for a positive constant $c\in\mathbb{R}$,
\begin{equation}\label{normequiv2}
|v|_{\mathbb{H}^s(\Omega_{\delta})}\leq\| v \|_{\mathbb{H}^s(\Omega_{\delta})}\leq c |v |_{\mathbb{H}^s(\Omega_{\delta})},
\end{equation}
then, we can renormize the space $\mathbb{H}_0^{\delta,s}(\Omega)$ and consider it endowed with the norm
\begin{equation}\label{normHHfrac}
\vertiii{v}_{\mathbb{H}_0^{\delta,s}}^2=\int_{\Omega_{\delta}}\int_{\Omega_{\delta}\cap B(x,\delta)}\frac{|v(x)-v(y)|^2}{|x-y|^{N+2s}}dydx.
\end{equation}
As a consequence, we have the following.
\begin{lemma}\label{lemmaHilbert}
The space $\mathbb{H}_0^{\delta,s}(\Omega)$ is a Hilbert space endowed with norm $\vertiii{\cdot}_{\mathbb{H}_0^{\delta,s}}$ induced by the scalar product
\begin{equation*}
\langle u,v\rangle_{\mathbb{H}_0^{\delta,s}}=\int_{\Omega_{\delta}}\int_{\Omega_{\delta}\cap B(x,\delta)}\frac{(u(x)-u(y))(v(x)-v(y))}{|x-y|^{N+2s}}dydx.
\end{equation*}
\end{lemma}

\begin{lemma}\label{lemmaconvergence}
Let $\{v_j\}_{j\in\mathbb{N}}$ be a bounded sequence in $\mathbb{H}_0^{\delta,s}(\Omega)$. Then, there exists $v\in L^p(\Omega)$ such that, up to a subsequence,
\begin{equation*}
v_j\to v\quad \text{in}\ L^p(\Omega),\ \text{as}\ j\to+\infty,
\end{equation*}
for any $p\in[2,2_s^*)$, where $2_s^*=\frac{2N}{N-2s}$ is the critical (fractional) Sobolev exponent.
\end{lemma} 
Thanks to Proposition \ref{belcor} and Lemma \ref{poincare}, the proofs of Lemma \ref{lemmaHilbert} and Lemma \ref{lemmaconvergence} follow similarly as to the case involving the standard fractional Laplacian $(-\Delta)_{\infty}^s$ and the space $\mathcal{H}_0^s(\Omega)$, (cf. \cite[Lemma 7]{SerVal2} and \cite[Lemma 8]{SerVal2} respectively).\newline
Next, to study convergence phenomena when one takes $\delta\to+\infty$, it will be essential to study the relation between the spaces $\mathcal{H}_0^s(\Omega)$ and $\mathbb{H}_0^{\delta,s}(\Omega)$. To that end, we show the following, whose proof is deferred to Section \ref{horizoninfty}.
\begin{lemma}\label{isomorfia} For any $\delta>0$, the spaces $\mathbb{H}_0^{\delta,s}(\Omega)$ and $\mathcal{H}_0^s(\Omega)$ are isomorphic. In particular, there exists a constant $C=C(\delta)>1$ such that 
\begin{equation*}
\vertiii{\cdot}_{\mathbb{H}_0^{\delta,s}}^2\leq\vertiii{\cdot}_{\mathcal{H}_0^s}^2\leq C(\delta)\vertiii{\cdot}_{\mathbb{H}_0^{\delta,s}}^2\quad\text{for all }\delta>0.
\end{equation*}
Moreover, $C(\delta)\to 1$ as $\delta \to +\infty$. 
\end{lemma}
Let us notice that Lemma \ref{isomorfia} implies the following convergence of spaces,
\begin{equation*}
\mathbb{H}_0^{\infty,s}(\Omega)\vcentcolon=\lim\limits_{\delta\to\infty}\mathbb{H}_0^{\delta,s}(\Omega)\equiv \mathcal{H}_0^s(\Omega).
\end{equation*}
Now we make precise the definition of weak solution of problem $(P_\delta^s)$.
\begin{definition}
We say that $u\in\mathbb{H}_0^{\delta,s}(\Omega)$ is a weak solution to problem \eqref{problema} if
\begin{equation*}
\frac{c_{N,s}}{2}\int_{\Omega_{\delta}}\int_{\Omega_{\delta}\cap B(x,\delta)}\frac{(u(x)-u(y))(v(x)-v(y))}{|x-y|^{N+2s}}dydx=\int_{\Omega}f(x)v(x)dx,
\end{equation*}
for all $v\in \mathbb{H}_0^{\delta,s}(\Omega)$, or written in a more compact way,
\begin{equation*}
\frac{c_{N,s}}{2}\langle u,v\rangle_{\mathbb{H}_0^{\delta,s}}=\langle f,v\rangle_{L^2(\Omega)},\quad\text{ for all } v \in \mathbb{H}_0^{\delta,s}(\Omega).
\end{equation*}
\end{definition}
Once we have introduced the functional setting we continue with some existence results dealing with the operator $(-\Delta)_{\delta}^s$.
\begin{theorem}\label{thlineal}
Given an horizon $\delta>0$ and a reaction term $f\in L^2(\Omega)$ there exists a unique solution $u^{\delta,s}$ to problem \eqref{problema}. Moreover, such a solution is the unique minimizer of the energy functional associated to \eqref{problema}, i.e.,
\begin{equation*}
\mathcal{J}_{\delta,s}(u^{\delta,s})=\min\limits_{v\in \mathbb{H}_0^{\delta,s}(\Omega)}\mathcal{J}_{\delta,s}(v)=\min\limits_{v\in \mathbb{H}_0^{\delta,s}(\Omega)}\left\{\frac{c_{N,s}}{4}\vertiii{v}_{\mathbb{H}_0^{\delta,s}}^2-\int_{\Omega}fv \, dx\right\}.
\end{equation*}
\end{theorem}
The proof of Theorem \ref{thlineal} is, thanks to the renormization of the space $\mathbb{H}_0^{\delta,s}(\Omega)$ in terms of $\vertiii{\cdot}_{\mathbb{H}_0^{\delta,s}}$ provided by Proposition \ref{belcor} and Lemma \ref{poincare}, an immediate consequence of Lax-Milgram Theorem, so we omit this proof.\newline
Next we study existence and stability issues for the eigenvalue problem 
\begin{equation}\label{eigenproblem}
     \left\{\begin{array}{rl}
     (-\Delta)_{\delta}^s\varphi=\lambda\varphi&\quad\mbox{in}\quad \Omega,\\
                         \varphi\!=0\ \          &\quad\mbox{on}\quad \partial_{\delta}\Omega.\\
\end{array}\right.
\tag{$EP_{\delta}^{s}$}
\end{equation}

\begin{proposition}\label{propeigen}
Let $\delta>0$, $s\in(0,1)$, $N>2s$ and $\Omega\subset\mathbb{R}^N$ an open bounded set with Lipschitz boundary. Then, the following hold:
\begin{enumerate}

\item Problem \eqref{eigenproblem} has a first positive eigenvalue that can be characterized as
\begin{equation}\label{minlambda1}
\lambda_{1}^{\delta,s}=\min\limits_{\substack{ u\in \mathbb{H}_0^{\delta,s}(\Omega)\\ \|u\|_{L^2(\Omega)}=1}}\frac{c_{N,s}}{2}\int_{\Omega_{\delta}}\int_{\Omega_{\delta}\cap B(x,\delta)}\frac{|u(x)-u(y)|^2}{|x-y|^{N+2s}}dydx,
\end{equation}
or, equivalenty,
\begin{equation*}
\lambda_{1}^{\delta,s}=\min\limits_{\substack{ u\in \mathbb{H}_0^{\delta,s}(\Omega)}}\frac{c_{N,s}}{2}\frac{\displaystyle\int_{\Omega_{\delta}}\int_{\Omega_{\delta}\cap B(x,\delta)}\frac{|u(x)-u(y)|^2}{|x-y|^{N+2s}}dydx}{\displaystyle\int_{\Omega}|u(x)|^2dx}.
\end{equation*}
Moreover, there exists a nonnegative function $\varphi_1^{\delta,s}\in\mathbb{H}_0^{\delta,s}(\Omega)$, which is an eigenfunction corresponding to $\lambda_1^{\delta,s}$, attaining the minimum in \eqref{minlambda1}, i.e., $\|\varphi_1^{\delta,s}\|_{L^2(\Omega)}=1$ and
\begin{equation}\label{minfirsteigenfunction}
\lambda_{1}^{\delta,s}=\frac{c_{N,s}}{2}\int_{\Omega_{\delta}}\int_{\Omega_{\delta}\cap B(x,\delta)}\frac{|\varphi_1^{\delta,s}(x)-\varphi_1^{\delta,s}(y)|^2}{|x-y|^{N+2s}}dydx.
\end{equation}
In addition, the first eigenvalue $\lambda_1^{\delta,s}$ is simple, i.e., if $\psi\in\mathbb{H}_0^{\delta,s}(\Omega)$ is such that, for all $\phi\in\mathbb{H}_0^{\delta,s}(\Omega)$,
\begin{equation}\label{simple}
\frac{c_{N,s}}{2}\langle \psi,\phi\rangle_{\mathbb{H}_0^{\delta,s}}=\lambda_1^{\delta,s}\int_{\Omega}\psi(x)\phi(x)dx,
\end{equation} 
then, $\psi=c\,\varphi_1^{\delta,s}$ with $c\in\mathbb{R}$.

\item The eigenvalues of \eqref{eigenproblem} are a countable set $\{\lambda_{k}^{\delta,s}\}_{k\in\mathbb{N}}$ satisfying
\begin{equation}\label{increasingeigenvalues}
0<\lambda_1^{\delta,s}<\lambda_2^{\delta,s}\leq\ldots\leq\lambda_k^{\delta,s}\leq\ldots
\end{equation}
and 
\begin{equation}\label{eigenvaluetoinfty}
\lambda_k^{\delta,s}\to +\infty\qquad\text{as } k\to +\infty.
\end{equation}
Furthermore, for any $k\in\mathbb{N},\ k\geq2$ the eigenvalues can be characterized as 
\begin{equation}\label{minlambdak}
\lambda_{k}^{\delta,s}=\min\limits_{\substack{ u\in \mathbb{P}_{k}^{\delta}\\ \|u\|_{L^2(\Omega)}=1}}\frac{c_{N,s}}{2}\int_{\Omega_{\delta}}\int_{\Omega_{\delta}\cap B(x,\delta)}\frac{|u(x)-u(y)|^2}{|x-y|^{N+2s}}dydx,
\end{equation}
or, equivalently,
\begin{equation*}
\lambda_{k}^{\delta,s}=\min\limits_{\substack{ u\in \mathbb{P}_{k}^{\delta}}}\frac{c_{N,s}}{2}\frac{\displaystyle\int_{\Omega_{\delta}}\int_{\Omega_{\delta}\cap B(x,\delta)}\frac{|u(x)-u(y)|^2}{|x-y|^{N+2s}}dydx}{\displaystyle\int_{\Omega}|u(x)|^2dx},
\end{equation*}
where
\begin{equation}\label{projectspace}
\mathbb{P}_{k}^{\delta}=\{u\in\mathbb{H}_0^{\delta,s}(\Omega):\langle u,\varphi_j^{\delta,s}\rangle_{\mathbb{H}_0^{\delta,s}}=0,\ j=1,\ldots,k-1\}.
\end{equation}
In addition, for any $k\in\mathbb{N},\ k\geq2$, there exists a function $\varphi_k^{\delta,s}\in \mathbb{P}_{k}^{\delta}$, which is an eigenfunction corresponding to $\lambda_k^{\delta,s}$, attaining the minimum in \eqref{minlambdak}, i.e., $\|\varphi_k^{\delta,s}\|_{L^2(\Omega)}=1$ and 
\begin{equation}\label{minksteigenfunction}
\lambda_k^{\delta,s}=\frac{c_{N,s}}{2}\int_{\Omega_{\delta}}\int_{\Omega_{\delta}\cap B(x,\delta)}\frac{|\varphi_k^{\delta,s}(x)-\varphi_k^{\delta,s}(y)|^2}{|x-y|^{N+2s}}dydx.
\end{equation}

\item The set of eigenfunctions $\{\varphi_k^{\delta,s}\}_{k\in\mathbb{N}}$ is an orthogonal basis of $\mathbb{H}_0^{\delta,s}(\Omega)$ and an orthonormal basis of $L^2(\Omega)$.

\item For any $k\in\mathbb{N}$, the eigenvalue $\lambda_k^{\delta,s}$ has finite multiplicity, i.e., if $\lambda_k^{\delta,s}$ satisfies
\begin{equation}\label{multiplicity}
\lambda_{k-1}^{\delta,s}<\lambda_k^{\delta,s}=\ldots=\lambda_{k+n}^{\delta,s}<\lambda_{k+n+1}^{\delta,s},
\end{equation}
for some $n\in\mathbb{N},\ n\geq1$, then the set of all the eigenfunctions corresponding to $\lambda_k^{\delta,s}$ belongs to
\begin{equation*}
span\{\varphi_k^{\delta,s},\ldots,\varphi_{k+n}^{\delta,s}\}.
\end{equation*}
In other words, denoting by $m_k^{\delta,s}$ the multiplicity of the eigenvalue $\lambda_k^{\delta,s}$, then
\begin{equation}\label{finitemultiplicity}
1\leq m_k^{\delta,s}<\infty,\quad\text{for all } k\in\mathbb{N}.
\end{equation}
\end{enumerate}
\end{proposition}
\begin{lemma}\label{boundinftyeigen}
Let $\varphi_k^{\delta,s}\in\mathbb{H}_0^{\delta,s}(\Omega)$ be an eigenfunction of \eqref{eigenproblem}, then $\varphi_k^{\delta,s}\in L^{\infty}(\Omega)$ for any $k\in\mathbb{N}$.
\end{lemma}
The proofs of Proposition \ref{propeigen} and Lemma \ref{boundinftyeigen} can be also done similarly as to the case dealing with the standard fractional Laplacian $(-\Delta)_{\infty}^s$ and the space $\mathcal{H}_0^s(\Omega)$. Indeed, up to minor modifications involving the joint use of Proposition \ref{belcor} and the fractional Sobolev inequality (cf.\cite[Theorem 6.5]{DiNezzaPalaValdi}), the proofs can be done following step by step the proofs of \cite[Proposition 9]{SerVal} and \cite[Proposition 4]{Servadei2013} respectively, so we omit the details for the sake of brevity.

We present now the main results of the work, dealing with the behavior of \eqref{problema} and \eqref{eigenproblem} when the horizon $\delta\to0^+$ and $\delta\to+\infty$. To that end, let us consider the following problems.
\begin{equation}\label{problemareescaled}
     \left\{\begin{array}{cl}
     (-\Delta)_{\delta}^{s}u=\frac{\delta^{2(1-s)}}{\kappa(N,s)}f &\quad\mbox{in}\quad \Omega,\\
              u=0 &\quad\mbox{on}\quad \partial_{\delta}\Omega,\\
		 \end{array}\right.
				\tag{$RP_{\delta}^{s}$}
\end{equation}
where 
\begin{equation}\label{constantkappa}
\kappa(N,s)=\frac{4N(1-s)}{\sigma_{N-1}c_{N,s}},
\end{equation}
with $\sigma_{N-1}$ the surface of the unitary sphere $\mathbb{S}^{N-1}$. As the limit $\delta\to0^+$ encodes a concentration phenomena, the rescaling on the reaction term $f$ of \eqref{problema}, namely $\frac{\delta^{2(1-s)}}{\kappa(N,s)}f$ in \eqref{problemareescaled}, is needed in order to avoid the degeneracy in the limit process.

\begin{remark} 
The constant $\kappa(N,s)$ is well defined when $s\to 1^-$. Indeed, cf. \cite{Bucur,Stinga}, we have 
\begin{equation}\label{limitecns}
c_{N,s}=\frac{2^{2s}s\Gamma(\frac{N}{2}+s)}{\pi^{\frac{N}{2}}\Gamma(1-s)}\to0\quad\text{as }s\to1^-,
\end{equation}
and, (cf. \cite[Corollary 4.2]{DiNezzaPalaValdi}),
\begin{equation}\label{limitecns2}
\lim\limits_{s\to 1^-}\frac{c_{N,s}}{(1-s)}=\lim\limits_{s\to 1^-}\frac{2^{2s}s\Gamma(\frac{N}{2}+s)}{\pi^{\frac{N}{2}}\Gamma(2-s)}=\frac{4\Gamma(\frac{N}{2}+1)}{\pi^{\frac{N}{2}}}=\frac{4N}{\frac{2\pi^{\frac{N}{2}}}{\Gamma(\frac{N}{2})}}=\frac{4N}{\sigma_{N-1}}.
\end{equation}
As a consequence,  
\begin{equation}\label{limitkapa}
\lim\limits_{s\to 1^-}\kappa(N,s)=1.
\end{equation}
\end{remark}
\noindent Next, let us consider the classical local linear problem,
\begin{equation}\label{problema0}
     \left\{\begin{array}{rl}
     (-\Delta)u=f &\quad\mbox{in}\quad \Omega,\\
              u=0 &\quad\mbox{on}\quad \partial\Omega,\\
		 \end{array}\right.
				\tag{$P_{0}^{1}$}
\end{equation}
and the linear problem driven by the standard fractional Laplacian,
\begin{equation}\label{problemainf}
     \left\{\begin{array}{rl}
     (-\Delta)_{\infty}^su=f &\quad\mbox{in}\quad \Omega,\\
                         u=0 &\quad\mbox{on}\quad \partial_{\infty}\Omega\equiv\mathbb{R}^{N}\backslash\Omega,\\
		 \end{array}\right.
				\tag{$P_{\infty}^{s}$}
\end{equation}
with $(-\Delta)_{\infty}^s$ defined in \eqref{fracclaplacian}. Our main results as regarding the relationship between the linear problems \eqref{problema}, \eqref{problemareescaled}, \eqref{problema0} and \eqref{problemainf} are the following.

\begin{theorem}\label{linear0}
Let $u^{\delta,s}$ and $u^{0,1}$ be the solutions of \eqref{problemareescaled} and \eqref{problema0} respectively. Then, up to a subsequence,
\begin{equation*}
u^{\delta,s}\to u^{0,1}\ \text{in }L^2(\Omega)\quad\text{as }\delta\to0^+. 
\end{equation*}
\end{theorem}

\begin{theorem}\label{linearinf}
Let $u^{\delta,s}$ and $u^{\infty,s}$ be the solutions of \eqref{problema} and \eqref{problemainf} respectively. Then, up to a subsequence,
\begin{equation*}
u^{\delta,s}\to u^{\infty,s}\ \text{in }L^2(\Omega)\quad\text{as }\delta\to+\infty.
\end{equation*}
\end{theorem}
On the other hand, let us consider the classical eigenvalue problem,
\begin{equation}\label{eigenproblem01}
     \left\{\begin{array}{rl}
     (-\Delta)\varphi=\lambda\varphi&\quad\mbox{in}\quad \Omega,\\
              \varphi=0\ \          &\quad\mbox{on}\quad \partial\Omega.
\end{array}\right.
\tag{$EP_{0}^{1}$}
\end{equation}
It is well known (cf. \cite{Brezis}) that the problem \eqref{eigenproblem01} has a countable set of eigenvalues that we denote by $\{\lambda_{k}^{0,1}\}_{k\in\mathbb{N}}$ and such that 
\begin{equation*}
0<\lambda_{1}^{0,1}<\lambda_{2}^{0,1} \leq\ldots\leq\lambda_{k}^{0,1}\leq\ldots ,\\
\end{equation*}
with
\begin{equation*}
\lambda_{k}^{0,1}\to+\infty\ \text{as}\ k\to+\infty,
\end{equation*} 
and, denoting by $m_k^{0,1}$ the multiplicity of the eigenvalue $\lambda_k^{0,1}$,
\begin{equation*}
1\leq m_k^{0,1}<\infty,\quad\text{for all }k\in\mathbb{N}.
\end{equation*}
Moreover, there exists a countable set of eigenfunctions $\{\varphi_k^{0,1}\}_{k\in\mathbb{N}}$ that is an orthogonal basis of $H_0^1(\Omega)$ and an orthonormal basis of $L^2(\Omega)$. The first eigenvalue is simple ($m_1^{0,1}=1$) and, by the Maximum Principle, $\varphi_1^{0,1}>0$ in $\Omega$. \newline 
Finally, let us also consider the nonlocal eigenvalue problem 
\begin{equation}\label{eigenprobleminfs}
     \left\{\begin{array}{rl}
     (-\Delta)_{\infty}^s\varphi=\lambda\varphi&\quad\mbox{in}\quad \Omega,\\
              \varphi=0\ \          &\quad\mbox{on}\quad \mathbb{R}^N\backslash\Omega.
\end{array}\right.
\tag{$EP_{\infty}^{s}$}
\end{equation}
with $(-\Delta)_{\infty}^s$ defined in \eqref{fracclaplacian}. Concerning this problem, Servadei and Valdinoci proved, cf. \cite{SerVal}, that \eqref{eigenprobleminfs} has a countable set of eigenvalues that we denote by $\{\lambda_{k}^{\infty,s}\}_{k\in\mathbb{N}}$ and such that
\begin{equation*}
0<\lambda_{1}^{\infty,s}<\lambda_{2}^{\infty,s}\leq\ldots\leq\lambda_{k}^{\infty,s}\leq\ldots,\\
\end{equation*}
with
\begin{equation*}
\lambda_{k}^{\infty,s}\to+\infty\ \text{as}\ k\to+\infty,
\end{equation*}
and, denoting by $m_k^{\infty,s}$ the multiplicity of the eigenvalue $\lambda_k^{\infty,s}$,
\begin{equation*}
1\leq m_k^{\infty,s}<\infty,\quad\text{for all }k\in\mathbb{N}.
\end{equation*}
Moreover, there exists a countable set of eigenfunctions $\{\varphi_k^{\infty,s}\}_{k\in\mathbb{N}}$ that is an orthogonal basis of 
\begin{equation*}
\mathbb{H}_0^{\infty,s}(\Omega)=\left\{v\in L^2(\Omega): \iint\limits_{\mathcal{D}}\frac{|v(x)-v(y)|^2}{|x-y|^{N+2s}}dydx<\infty,\ v=0 \text{ a.e. } \text{on } \mathbb{R}^N\backslash\Omega\right\},
\end{equation*}
and an orthonormal basis of $L^2(\Omega)$. Again, the first eigenvalue is simple ($m_1^{\infty,s}=1$) and $\varphi_1^{\infty,s}$ can chosen so that $\varphi_1^{\infty,s}\geq0$ in $\Omega$.\newline 
Finally, we relate the eigenvalues and eigenfunctions of \eqref{eigenproblem} to those of the eigenvalue problems \eqref{eigenproblem01} and \eqref{eigenprobleminfs} through the following results. 

\begin{theorem}\label{theigen}
Let $\{(\lambda_{k}^{\delta,s},\varphi_k^{\delta,s})\}_{k\in\mathbb{N}}$ be the set of eigenvalues and eigenfunctions of $(-\Delta)_{\delta}^s$ with homogeneous Dirichlet boundary condition on $\partial_{\delta}\Omega$ and let\break $\{(\lambda_{k}^{0,1},\varphi_k^{0,1})\}_{k\in\mathbb{N}}$ be the set of eigenvalues of $(-\Delta)$ with homogeneous Dirichlet boundary condition on $\partial\Omega$. Then,
\begin{equation*}
\kappa(N,s)\frac{\lambda_{k}^{\delta,s}}{\delta^{2(1-s)}}\to\lambda_{k}^{0,1}\quad \text{as}\ \delta\to0^+,
\end{equation*}
for the constant $\kappa(N,s)$ appearing in \eqref{constantkappa}, and there exists a subsequence (that we do not relabel) such that
\begin{equation*}
\varphi_{k}^{\delta,s}\to\varphi_{k}^{0,1}\ \text{in}\ L^2(\Omega)\quad \text{as}\ \delta\to0^+,
\end{equation*}
for every $k\in\mathbb{N}$. As a consequence, $m_k^{\delta,s}\to m_k^{0,1}$ as $\delta \to 0^+$, for any $k\ge 1$.
\end{theorem}

\begin{remark} We would like to emphasize that, as Theorems \ref{linear0} and \ref{theigen} shows, even though the fractionality parameter $s$ keeps fixed, the local problem driven by $(-\Delta)$ is recovered, under the appropriate rescaling, as $\delta\to 0^+$. A different issue is the limit as $s\to 1^-$. This has been addressed in the case of the fractional Laplacian in several references, both operator convergence (cf. \cite{DiNezzaPalaValdi}) and spectral stability (cf. \cite{BrascoPariniSquassina}). Let us see now that the same behavior holds for $(-\Delta)_\delta^s$ as $s\to 1^-$. 

First of all, because of \cite[Proposition 4.4]{DiNezzaPalaValdi}, for any $u\in C_0^{\infty}(\mathbb{R}^N)$, we have the pointwise convergence
\begin{equation*}
\lim\limits_{s\to 1^-}(-\Delta)_{\infty}^su(x)=(-\Delta)u(x).
\end{equation*} 
Indeed, the hypothesis $u\in C_0^{\infty}(\mathbb{R}^N)$ can be relaxed to $u\in\mathcal{C}^2(B(x,R))\cap L^{\infty}(\mathbb{R}^N)$ for some $R>0$, cf. \cite[Proposition 5.3]{Stinga}. Following \cite[Proposition 4.4]{DiNezzaPalaValdi}, by means of a change of variable in \eqref{fracclaplacian}, we have
\begin{equation*}
(-\Delta)_{\infty}^su(x)=-\frac{c_{N,s}}{2}\int_{\mathbb{R}^N}\frac{u(x+y)+u(x-y)-2u(x)}{|y|^{N+2s}}dy,
\end{equation*}
so that, for $u\in L^{\infty}(\mathbb{R}^N)$, 
\begin{equation*}
\begin{split}
\left|\int_{\mathbb{R}^N\backslash B(0,\delta)}\mkern-10mu\frac{u(x+y)+u(x-y)-2u(x)}{|y|^{N+2s}}dy\right|&\leq 4\|u\|_{L^{\infty}(\mathbb{R}^N)}\int_{\mathbb{R}^N\backslash B(0,\delta)}\frac{1}{|y|^{N+2s}}dy\\
&\leq4\sigma_{N-1}\|u\|_{L^{\infty}(\mathbb{R}^N)}\int_{\delta}^{\infty}\frac{1}{\rho^{2s+1}}d\rho\\
&=\frac{2\sigma_{N-1}}{s\delta^{2s}}\|u\|_{L^{\infty}(\mathbb{R}^N)}.
\end{split}
\end{equation*}
Thus, by \eqref{limitecns}, we find that, for any $\delta>0$ fixed,
\begin{equation*}
\lim\limits_{s\to 1^-}\left|\frac{c_{N,s}}{2}\int_{\mathbb{R}^N\backslash B(0,\delta)}\mkern-10mu\frac{u(x+y)+u(x-y)-2u(x)}{|y|^{N+2s}}dy\right|=0.
\end{equation*}
Therefore, recalling \eqref{seconddiff},
\begin{equation}\label{limitesto1}
\begin{split}
\lim\limits_{s\to 1^-}(-\Delta)_{\infty}^su(x)&=\lim\limits_{s\to 1^-}-\frac{c_{N,s}}{2}\int_{B(0,\delta)}\frac{u(x+y)+u(x-y)-2u(x)}{|y|^{N+2s}}dy\\
&=\lim\limits_{s\to 1^-}(-\Delta)_{\delta}^su(x).
\end{split}
\end{equation}
We conclude that, for any $\delta>0$ fixed and a given $u\in L^{\infty}(\mathbb{R}^N)$, the behavior of $(-\Delta)_\infty^su(x)$ and $(-\Delta)_\delta^su(x)$ as $s\to 1^-$ coincides. Hence, under sufficient smothness asumptions,
\begin{equation*}
\lim\limits_{s\to 1^-}(-\Delta)_{\delta}^su(x)=(-\Delta)u(x)\quad\text{for all }\delta>0\text{ fixed}.
\end{equation*}
On the other hand, because of \cite[Proposition 4]{Servadei2013}, the eigenfunctions of the fractional Laplacian $(-\Delta)_{\infty}^s$ are bounded, i.e., $\varphi_k^{\infty,s}\in L^{\infty}(\mathbb{R}^N)$ for all $k\in\mathbb{N}$. Moreover, these eigenfunctions are H\"older continuous, cf. \cite{Servadei2014,RosOton2012}. Also, because of Lemma \ref{boundinftyeigen}, we have $\varphi_k^{\delta,s}\in L^{\infty}(\mathbb{R}^N)$ for all $k\in\mathbb{N}$. Therefore, using \eqref{limitesto1} it is not difficult to show that,
for any $\delta>0$ fixed,
\begin{equation}\label{limitequal}
\lim\limits_{s\to 1^-}\lambda_k^{\delta,s}=\lim\limits_{s\to 1^-}\lambda_k^{\infty,s}\quad\text{for all }k\in\mathbb{N}.
\end{equation}
Next, by \cite[Theorem 1.2]{BrascoPariniSquassina} with $p=2$,
\begin{equation}\label{rem}
\lim\limits_{s\to 1^-}\lambda_{k}^{\infty,s}=\lambda_{k}^{0,1}.
\end{equation}
Observe that no normalization constant is used in \cite{BrascoPariniSquassina} to define the fractional operator, just a factor $2$ instead of $c_{N,s}$. On the other hand, in \cite[Theorem 1.2]{BrascoPariniSquassina}  the limit process is normalized by a factor $(1-s)$, so that adjusting constants and taking in mind \eqref{limitecns2} and \eqref{contbr} below, equality \eqref{rem} follows accordingly to our normalization setting. Finally, because of \eqref{limitkapa}, \eqref{limitequal} and \eqref{rem},
\begin{equation*}
\lim\limits_{s\to 1^-}\kappa(N,s)\frac{\lambda_{k}^{\delta,s}}{\delta^{2(1-s)}}=\lim\limits_{s\to 1^-}\lambda_{k}^{\delta,s}=\lim\limits_{s\to1^-}\lambda_{k}^{\infty,s}=\lambda_k^{0,1}.
\end{equation*}
\end{remark}

\begin{theorem}\label{theigen2}
Let $\{(\lambda_{k}^{\delta,s},\varphi_{k}^{\delta,s})\}_{k\in\mathbb{N}}$ be the set of eigenvalues and eigenfunctions of $(-\Delta)_{\delta}^s$ with homogeneous Dirichlet boundary condition on $\partial_{\delta}\Omega$ and let\break $\{(\lambda_{k}^{\infty,s},\varphi_{k}^{\infty,s})\}_{k\in\mathbb{N}}$ be the set of eigenvalues of $(-\Delta)_{\infty}^{s}$ with homogeneous Dirichlet boundary condition on $\mathbb{R}^N\backslash\Omega$. Then,
\begin{equation*}
\lambda_{k}^{\delta,s}\to\lambda_{k}^{\infty,s}\quad \text{as}\ \delta\to+\infty,
\end{equation*}
and there exists a subsequence (that we do not relabeled) such that
\begin{equation*}
\varphi_{k}^{\delta,s}\to\varphi_{k}^{\infty,s}\ \text{in}\ L^2(\Omega)\quad \text{as}\ \delta\to+\infty,
\end{equation*}
for every $k\in\mathbb{N}$. As a consequence, $m_k^{\delta,s}\to m_k^{\infty,s}$ as $\delta \to \infty$, for any $k\ge 1$.
\end{theorem}
An interesting question, out of the scope of this paper, is whether the spectral structure of $(-\Delta)$ and $(-\Delta)_{\infty}^s$ coincides. In other words, whether the multiplicity $m_k^{\delta,s}$ remains constant as a function of the horizon $\delta>0$. It seems a very hard task to solve, since not so much is known about the multiplicity of eigenvalues, even for the classical Laplace operator $(-\Delta)$. Nevertheless, because of the above results, $m_k^{0,1}$ and $m_k^{\delta,s}$ coincide, at least for $\delta>0$ small enough, and $m_k^{\infty,s}$ and $m_k^{\delta,s}$ coincide, at least, for $\delta>0$ big enough. Roughly speaking, for $\delta<<1$ the operator $(-\Delta)_{\delta}^s$ is almost a local operator, in the sense that it takes into account what happens in a neighborhood of size $\delta$ while for $\delta>>1$ the operator $(-\Delta)_{\delta}^s$ takes into account almost the same as $(-\Delta)_{\infty}^s$.


\section{Preliminary results: $\Gamma$-convergence and localization}\label{preliminaryresults}
This section includes some results that play a crucial role to study problems \eqref{problemareescaled} and \eqref{eigenproblem} when we take the limit $\delta\to0^+$. We start with the classical {\it localization} result of Bourgain, Brezis and Mironescu, cf. \cite{BourBrezMiro}, that provides us with a first hint towards Theorem \ref{theigen}. Let $\{\rho_{n}(x)\}_{n\in\mathbb{N}}$ be a sequence of radial mollifiers, i.e.
\begin{equation*}
\rho_n(x)=\rho_n(|x|),\ \rho_n(x)\geq0,\ \int\rho_n(x)dx=1
\end{equation*}
satisfying
\begin{equation*}
\lim\limits_{n\to\infty}\int_{\varepsilon}^{\infty}\rho_n(r)r^{N-1}=0,\ \text{for every}\ \varepsilon>0
\end{equation*}
\begin{theorem}\label{BoBrMi}\cite[Theorem 2]{BourBrezMiro} Assume $u\in L^p(\Omega)$, $1<p<\infty$. Then, for a constant $C=C(N,p)>0$, we have
\begin{equation*}
\lim\limits_{n\to\infty}\int_{\Omega}\int_{\Omega}\frac{|u(x)-u(y)|^p}{|x-y|^p}\rho_n(x-y)dydx=C\int_{\Omega}|\nabla u|^pdx.
\end{equation*}
with the convention that $\int_{\Omega}|\nabla u|^pdx=\infty$ if $u\notin W^{1,p}(\Omega)$.
\end{theorem}

\subsection{$\Gamma$-Convergence}\hfill\newline 
The limit process in the sense of $\Gamma$-convergence, denoted by $\overset{\Gamma}{\to}$, is the right concept of limit for variational problems since it, together with equicoercivity or compactness, implies that minimizers of $I_{\delta}$ converge to minimizers of $I$ as well as their energies. A nice account on $\Gamma$-convergence is provided in \cite{Braides}. We present now a result about $\Gamma$-convergence of functionals proved in \cite{BellCorPed} that is the core of the proof of Theorem \ref{linear0} and Theorem \ref{theigen}. Let us consider a functional of the form 
\begin{equation}\label{funcional}
\int_{\Omega}\int_{\Omega\cap B(x,\delta)}\omega(x-y,u(x)-u(y))dydx,
\end{equation}
with $u:\Omega\mapsto\mathbb{R}$. In our case the \textit{potential function} $\displaystyle\omega(\overline{x},\overline{y})=\frac{|\overline{y}|^2}{|\overline{x}|^{N+2s}}$ fits into the cases studied in \cite{BellCorPed}, where the type of functional considered orbits around the important case
\begin{equation*}
\omega(\overline{x},\overline{y})=\frac{|\overline{y}|^p}{|\overline{x}|^{\alpha}},\quad \text{where}\ 1<p<+\infty,\,0\leq\alpha<N+p,
\end{equation*}
related to the fractional $p\,$-Laplacian. In particular, the result establishes that, under some hypotheses to be stated below and for some $\beta>0$ to be specified, the sequence of rescaled functionals
\begin{equation}\label{scaledfunctional}
I_{\delta}(u)=\frac{N+\beta}{\delta^{N+\beta}}\int_{\Omega}\int_{\Omega\cap B(x,\delta)}\omega(x-y,u(x)-u(y))dydx.
\end{equation}
has a $\Gamma$-limit (under the strong-$L^p(\Omega)$ topology) given by the local functional 
\begin{equation*}
I_{\delta}(u)\overset{\Gamma}{\to}I_0(u)=\int_{\Omega}W(\nabla u)dx,\quad \text{as}\ \delta\to0^+,
\end{equation*}
for some function $W(\cdot)$ that we construct next. Specifically, the $\Gamma$-limit is recovered in several steps:
\begin{enumerate}

\item \textit{Scaling}: we scale the functional \eqref{funcional} to obtain the functional $I_{\delta}(u)$ given in \eqref{scaledfunctional} with $\beta$ given in the next step.

\item \textit{Blow-up at zero}: we assume that there exists some $\beta\in\mathbb{R}$ such that the following limit exists,
\begin{equation}\label{blowup}
\omega^{\circ}(\overline{x},\overline{y})=\lim\limits_{t\to0^+}\frac{1}{t^{\beta}}\omega(t\overline{x},t\overline{y}).
\end{equation}

\item \textit{Limit density}: the limit density $\overline{\omega}:\Omega\times\mathbb{R}^{1\times N}\mapsto\mathbb{R}$  is given by 
\begin{equation}\label{limitdensity}
\overline{\omega}(F)\vcentcolon=\int_{\mathbb{S}^{N-1}}\omega^{\circ}(z,Fz)d\sigma(z),
\end{equation}
where $\mathbb{S}^{N-1}$ is the $(N-1)$-dimensional unitary sphere.
\item \textit{Convexification}: The $\Gamma$-limit of $I_{\delta}$ as $\delta\to0^+$ is given by
\begin{equation}\label{quasiconvex}
I_0(u)=\int_{\Omega}\overline{\omega}^{c}(\nabla u)dx,
\end{equation}
where,
\begin{equation*}
\overline{\omega}^{c}(F)\vcentcolon=\sup\{v(F):v(\cdot)\leq \overline{\omega}(\cdot)\ \text{and}\ v(\cdot)\ \text{convex}\}.
\end{equation*}
\end{enumerate}
The right scaling required in \eqref{blowup} for the \textit{potential function}
\begin{equation*}
\omega(\tilde{x},\tilde{y})=\frac{|\tilde{y}|^p}{|\tilde{x}|^\alpha}
\end{equation*}
is then given by $\beta=p-\alpha$.\newline 
The key result we use to prove the main results of the paper is the following, which is a simplified version of \cite[Theorem 1]{BellCorPed} but enough for our aim here. Let us denote $\tilde{\Omega}\vcentcolon=\{ z=x-y\;:\; x,\,y\in \Omega\}$.

\begin{theorem}\cite[Theorem 1]{BellCorPed}\label{belcorped} Let $\Omega\subset\mathbb{R}^N$ be a bounded domain with Lipschitz boundary and $\omega:\Omega\times\tilde{\Omega}\times\mathbb{R}\mapsto\mathbb{R}$ satisfying the hypotheses (H1)-(H5) below. 
\begin{itemize}
\item[a)]{\it Compactness:} For each $\delta>0$ let 
\begin{equation*}
u_{\delta}\in \mathcal{A}_{\delta}\vcentcolon=\{v\in L^p(\Omega): v=0\ \text{on}\ \partial_{\delta}\Omega\}\quad \text{with}\ \sup\limits_{\delta>0} I_{\delta}(u_{\delta})<+\infty.
\end{equation*}
Then, there exist 
\begin{equation*}
u\in\mathcal{B}\vcentcolon=\{v\in W^{1,p}(\Omega): v=0\ \text{on}\ \partial\Omega\},
\end{equation*}
such that,
\begin{equation*}
u_{\delta}\to u\ \text{in}\ L^p(\Omega)\quad \text{as}\ \delta\to0^+.
\end{equation*}
\item[b)]{\it Liminf inequality:} For each $\delta>0$ let $u_{\delta}\in\mathcal{A}_{\delta}$ and $u\in\mathcal{B}$ such that $u_{\delta}\to u$ in $L^p(\Omega)$ as $\delta\to0^+$. Then,
\begin{equation*}
I_0(u)\leq \liminf\limits_{\delta\to0^+}I_{\delta}(u_{\delta}).
\end{equation*}
\item[c)]{\it Limsup inequality:} For each $\delta>0$ and $u\in\mathcal{B}$ there exist $u_{\delta}\in\mathcal{A}_{\delta}$, called \textit{recovery sequence}, such that $u_{\delta}\to u$ in $L^p(\Omega)$ as $\delta\to0^+$ and 
\begin{equation*}
\limsup\limits_{\delta\to0^+}I_{\delta}(u_{\delta})\leq I_0(u).
\end{equation*}
\end{itemize}
\end{theorem}
For a general potential function $\omega(\overline{x},\overline{y})$, the hypotheses required in Theorem \ref{belcorped} are quite involved but, as it is noted in \cite{BellCorPed}, for a potential function of the form
\begin{equation*}
\omega(\tilde{x},\tilde{y})=f(\tilde{x})g(\tilde{y}),
\end{equation*} 
the necessary hypotheses are the following:
\begin{itemize}
\item[H1)] $f$ is Lebesgue measurable and $g$ is Borel measurable.
\item[H2)] $g$ is convex.
\item[H3)] There exists $c_0>0$ such that 
\begin{equation*}
\frac{c_0}{|\tilde{x}|^{\alpha}}\leq f(\tilde{x})\quad \text{and}\quad c_0|\tilde{y}|^p\leq g(\overline{y})\qquad\text{for}\ \tilde{x}\in\widetilde{\Omega},\ \tilde{y}\in\mathbb{R}.
\end{equation*}
\item[H4)] There exists $c_1>0$ and $h\in L^1(\mathbb{S}^{N-1})$ with $h\geq0$ such that 
\begin{equation*}
f(\tilde{x})\leq h\left(\frac{\tilde{x}}{|\tilde{x}|}\right)\frac{1}{|\tilde{x}|^{\alpha}}\ \ \text{and}\ \ g(\tilde{y})\leq c_1 |\tilde{y}|^p\qquad\text{for}\ \tilde{x}\in\widetilde{\Omega},\ \tilde{y}\in\mathbb{R}.
\end{equation*}
\item[H5)] The functions $f^{\circ}:\mathbb{R}\backslash\{0\}\mapsto\mathbb{R}$ and $g^{\circ}:\mathbb{R}\mapsto\mathbb{R}$ defined as
\begin{equation*}
f^{\circ}(\tilde{x})\vcentcolon=\lim\limits_{t\to0^+}t^{\alpha}f(t\tilde{x})\quad\text{and}\quad g^{\circ}(\tilde{y})\vcentcolon=\lim\limits_{t\to0^+}\frac{1}{t^p}g(t\tilde{y}),
\end{equation*}
are continuous and, for each compact $K\subset\mathbb{R}$,
\begin{equation*}
\lim\limits_{t\to0^+}\sup\limits_{\tilde{x}\in\mathbb{S}^{N-1}}|t^{\alpha}f(t\tilde{x})-f^{\circ}(\tilde{x})|=0\quad\text{and}\quad \lim\limits_{t\to0^+}\sup\limits_{K\subset\mathbb{R}}|\frac{1}{t^p}g(t\tilde{y})-g^{\circ}(\tilde{y})|=0.
\end{equation*}
\end{itemize}

\vspace{0.2cm}

A straightforward consequence, cf. \cite{Braides}, of $\Gamma$-convergence and the compactness property (Theorem \ref{belcorped}-b) and c) and Theorem \ref{belcorped}-a) respectively), is the following corollary. Notice that under previous hypothesis existence of minimizers for $I_\delta$ is guaranteed, cf. \cite{BellCor}. 

\begin{corollary}\label{corGammaconv} In the conditions of Theorem \ref{belcorped}, let $u_\delta\in\mathbb{H}_0^{\delta,s}(\Omega)$ be a minimizer of $I_\delta$, for any $\delta>0$. Then, there exists $u_0\in H_0^1(\Omega)$ a minimizer of $I_0$ such that, up to a subsequence,
\begin{equation*}
u_\delta\to u_0\mbox{  strong in  }L^2(\Omega)\quad\text{as }\delta\to 0^+,
\end{equation*}
and 
\begin{equation*}
I_\delta(u_\delta)\to I_0(u_0)\quad\text{as }\delta\to 0^+.
\end{equation*} 
\end{corollary} 

\section{Taking the horizon $\delta\to0^+$}\label{horizon0}
In this section we prove Theorem \ref{linear0} and Theorem \ref{theigen}. To that end, we prove first a preliminary result concerning the $\Gamma$-convergence.
\begin{lemma}\label{Glimit} Let us consider the scaled functional 
\begin{equation}\label{funcionaldeltas}
I_{\delta,s}(u)=\frac{2(1-s)}{\delta^{2(1-s)}}\int_{\Omega_{\delta}}\int_{\Omega_{\delta}\cap B(x,\delta)}\frac{|u(x)-u(y)|^2}{|x-y|^{N+2s}}dydx.
\end{equation} 
defined on $\mathbb{H}_0^{\delta,s}(\Omega)$. Then, the $\Gamma$-limit of $I_{\delta,s}(u)$ as $\delta\to 0^+$ is given by
\begin{equation}\label{GlimitFunc}
I_0(u)=\gamma\int_{\Omega}|\nabla u(x)|^2dx,
\end{equation}
for a constant $\gamma=\gamma(N)=\frac{\sigma_{N-1}}{N}>0$, being $\sigma_{N-1}$ the measure of the $(N-1)$-dimensional unit sphere.
\end{lemma}

\begin{proof}
Let us follow the construction process of the limit density as specified in \eqref{scaledfunctional}-\eqref{quasiconvex}. First, let us observe that, since the potential function 
\begin{equation*}
\omega(\tilde{x},\tilde{y})=\frac{|\tilde{y}|^2}{|\tilde{x}|^{N+2s}},
\end{equation*}
then $\omega(\tilde{x},\tilde{y})=f(\tilde{x})g(\tilde{y})$ with $f(\tilde{x})=\frac{1}{|\tilde{x}|^{N+2s}}$ and $g(\tilde{y})=|\tilde{y}|^2$. Hence, in our case, $p=2$, $\alpha=N+2s<N+2$ and hypotheses (H1)-(H5) are clearly satisfied. Next we construct the function given in \eqref{blowup}. Since $\omega(\tilde{x},\tilde{y})=\frac{|\tilde{y}|^2}{|\tilde{x}|^{N+2s}}$ is homogeneous, it is immediate that, taking $\beta=2-(N+2s)$ in \eqref{blowup},
\begin{equation*}
\begin{split}
\omega^{\circ}(\tilde{x},\tilde{y})&=\lim\limits_{t\to0^+}\frac{1}{t^\beta}\omega(t\tilde{x},t\tilde{y})=\lim\limits_{t\to0^+}\frac{1}{t^{2-(N-2s)}}\frac{t^2|\tilde{y}|^2}{t^{N+2s}|\tilde{x}|^{N+2s}}=\frac{|\tilde{y}|^2}{|\tilde{x}|^{N+2s}}\\
&=\omega(\tilde{x},\tilde{y}).
\end{split}
\end{equation*}
We continue by constructing the function $\overline{\omega}(F)$ given in \eqref{limitdensity}. Because of 
\begin{equation*}
\omega^{\circ}(\tilde{x},\tilde{y})=\omega(\tilde{x},\tilde{y})=\omega(|\tilde{x}|,|\tilde{y}|),
\end{equation*}
we find that, for a given vector $F\in\mathbb{R}^{1\times N}$, the function $\overline{\omega}(F):\Omega\times\mathbb{R}^{1\times N}\mapsto\mathbb{R}$ is given by
\begin{equation*}
\begin{split}
\overline{\omega}(F)&=\int_{\mathbb{S}^{N-1}}\omega^{\circ}(z,F\cdot z)d\sigma=\int_{\mathbb{S}^{N-1}}\omega(|z|,|F\cdot z|)d\sigma=\int_{\mathbb{S}^{N-1}}\omega(1,|F\cdot z|)d\sigma\\
&=\int_{\mathbb{S}^{N-1}}|F\cdot z|^2d\sigma=|F|^2\int_{\mathbb{S}^{N-1}}\frac{|F\cdot z|^2}{|F|^2}d\sigma=|F|^2\int_{\mathbb{S}^{N-1}}\left|\frac{F}{|F|}\cdot z\right|^2d\sigma\\
&=|F|^2\int_{\mathbb{S}^{N-1}}|e\cdot z|^2d\sigma,
\end{split}
\end{equation*}
with $e=\frac{F}{|F|}\in\mathbb{S}^{N-1}$ an unitary vector. Moreover, it is easy to see that, given a unitary vector $e\in\mathbb{S}^{N-1}$,
\begin{equation}\label{contbr}
\int_{\mathbb{S}^{N-1}}|e\cdot z|^2d\sigma=\frac{\sigma_{N-1}}{N},
\end{equation}
where $\sigma_{N-1}$ is the area of the unitary sphere $\mathbb{S}^{N-1}$. Hence, we conclude
\begin{equation*}
\overline{\omega}(F)=\frac{\sigma_{N-1}}{N}|F|^2=\gamma(N)|F|^2.
\end{equation*}
At last, since the function $\overline{\omega}(F)=\gamma(N) |F|^2$ is convex, the convexification $\overline{\omega}^{c}(F)$ coincides with the function $\overline{\omega}(F)$ itself. Therefore, we conclude that the $\Gamma$-limit of the functional $I_{\delta,s}(\cdot)$ as $\delta\to 0^+$ is given by
\begin{equation*}
I_0(u)=\int_{\Omega}\overline{\omega}^{c}(\nabla u(x))dx=\int_{\Omega}\overline{\omega}(\nabla u(x))dx=\gamma(N) \int_{\Omega}|\nabla u(x)|^2dx.
\end{equation*}
\end{proof}
At this point it is worth mentioning that, up to the correct scaling $\displaystyle\frac{2(1-s)}{\delta^{2(1-s)}}$, all the functionals $I_{\delta,s}$ will $\Gamma$-converge to the same limit independently of $s$, because, in the previous proof, the denominator of $\omega^{\circ}(\overline{x},\overline{y})$ plays no role in the integration over the unit sphere $\mathbb{S}^{N-1}$.
\begin{proof}[Proof of Theorem \ref{linear0}]\hfill\newline 
Let $\mathcal{F}(u)$ be the energy functional associated to \eqref{problemareescaled}, i.e.,
\begin{equation*}
\mathcal{F}(u)=\frac{c_{N,s}}{4}\int_{\Omega_{\delta}}\int_{\Omega_{\delta}\cap B(x,\delta)}\frac{|u(x)-u(y)|^2}{|x-y|^{N+2s}}dydx-\frac{\delta^{2(1-s)}}{\kappa(N,s)}\int_{\Omega}fudx,
\end{equation*}
so that the solution $u^{\delta,s}$ to \eqref{problemareescaled} is a minimum of $\mathcal{F}(u)$. Next, let us define $\kappa(N,s)\vcentcolon=\frac{4(1-s)}{c_{N,s}\gamma(N)}$, with $\gamma(N)$ the constant given in Lemma \ref{Glimit} and consider the rescaled functional 
\begin{equation*}
\begin{split}
\mathcal{F}_{\delta,s}(u)&\vcentcolon=\frac{1}{\delta^{2(1-s)}}\mathcal{F}(u)\\
												&=\frac{c_{N,s}}{4(1-s)}\frac{(1-s)}{\delta^{2(1-s)}}\int_{\Omega_{\delta}}\int_{\Omega_{\delta}\cap B(x,\delta)}\frac{|u(x)-u(y)|^2}{|x-y|^{N+2s}}dydx-\frac{1}{\kappa(N,s)}\int_{\Omega}fudx\\
&=\frac{1}{\kappa(N,s)}\left(\frac{1}{2\gamma(N)}I_{\delta,s}(u)-\int_{\Omega}fudx\right).
\end{split}
\end{equation*}
Clearly, if $u^{\delta,s}$ is a solution to \eqref{problemareescaled}, or equivalently, a minimizer of $\mathcal{F}(u)$, then it is also a minimizer of the rescaled functional $\mathcal{F}_{\delta,s}$. As we see in the proof of Lemma \ref{Glimit}, the hypotheses (H1)-(H5) of Theorem \ref{belcorped} are satisfied and the $\Gamma$-limit of $I_{\delta,s}(u)$ as $\delta\to 0^+$ is given by \eqref{GlimitFunc}. Hence, the $\Gamma$-limit of $\mathcal{F}_{\delta,s}$ as $\delta\to 0^+$ is 
\begin{equation*}
\begin{split}
\mathcal{F}_0(u)&=\frac{1}{\kappa(N,s)}\left(\frac{1}{2} I_0(u)-\int_{\Omega}fudx\right)\\
								&=\frac{1}{\kappa(N,s)}\left(\frac{1}{2} \int_{\Omega}|\nabla u|^2dx-\int_{\Omega}fudx\right).
\end{split}
\end{equation*}
We conclude that the $\Gamma$-limit functional is, up to a constant, the energy functional associated to \eqref{problema0}. On the other hand, as $\Gamma$-convergence implies convergence of optimal energies, if $u^{\delta,s}\in \mathbb{H}_0^{\delta,s}(\Omega)$ is the minimizer of $\mathcal{F}_{\delta,s}$ and $u^{0,1}$ the minimizer of $\mathcal{F}_0$, then 
\begin{equation*}
\lim_{\delta\to 0^+} \mathcal{F}_{\delta,s}(u^{\delta,s})=\mathcal{F}_{0}(u^{0,1}),
\end{equation*}
Indeed, if $u^{\delta,s}\in \mathbb{H}_0^{\delta,s}(\Omega)$ is the minimizer of $\mathcal{F}_{\delta,s}$, so that $u^{\delta,s}$ is a solution to problem \eqref{problemareescaled}, we have
\begin{equation*}
\mathcal{F}_{\delta,s}(u_\delta)=-\frac{1}{2\kappa(N,s)} I_{\delta,s}(u^{\delta,s}).
\end{equation*}
Therefore, the sequence $\{I_{\delta,s}(u^{\delta,s})\}_{\delta>0}$ is bounded and, by Theorem \ref{belcorped}, there exists $u^{0,1}\in H_0^1(\Omega)$ such that, up to a subsequence $u^{\delta,s}\to u^{0,1}$ strong in $L^2(\Omega)$. Moreover, $\Gamma$-convergence of $\mathcal{F}_{\delta,s}$ to $\mathcal{F}_0$ implies that $u^{0,1}$ is the unique minimizer of $\mathcal{F}_0$, and therefore the solution of $(P_0^1)$. 
\end{proof}

\begin{proof}[Proof of Theorem \ref{theigen}]\hfill\newline
First we prove the result for the sequence of first eigenvalues and corresponding eigenfunctions. Recall that first eigenvalues, $\lambda_1^{\delta,s}$ and $\lambda_1^{0,1}$, are simple for both $(EP_\delta^s)$, for any $\delta>0$, and for the limit problem $(EP_0^1)$ respectively. Let $I_\delta^{(1)}$ and $I^{(1)}$ be the restricted functionals
\begin{equation*}
I_{\delta,s}^{(1)}(u)=\left\{ 
\begin{array}{ll} 
I_{\delta,s}(u) &\mbox{if  }\|u\|_{L^2(\Omega)}=1,\\
+\infty &\mbox{otherwise},
\end{array}\right.
\end{equation*}
and
\begin{equation*}
I^{(1)}(u)=\left\{ 
\begin{array}{ll} 
I(u) &\mbox{if  }\|u\|_{L^2(\Omega)}=1,\\
+\infty &\mbox{otherwise}, 
\end{array}\right.
\end{equation*}
respectively, with $I_{\delta,s}(u)$ defined in \eqref{funcionaldeltas} and $I(u)=\int_\Omega |\nabla u|^2\,dx$. Let us show that $I_{\delta,s}^{(1)}\overset{\Gamma}{\to} I^{(1)}$. Actually, this is derived easily from Lemma \ref{Glimit} and Theorem \ref{belcorped}:
\begin{enumerate}
\item {\it Liminf inequality:} Given $u_\delta\to u$ strong in $L^2(\Omega)$, then
\begin{equation*}
I^{(1)}(u)\le \liminf_{\delta\to 0^+} I_{\delta,s}^{(1)}(u_\delta).
\end{equation*} 
this is consequence of Theorem \ref{belcorped}-b), and the fact that strong convergence in $L^2(\Omega)$ implies convergence of the norms. 
\item {\it Limsup inequality:} Given $u\in H_0^1(\Omega)$, with $\|u\|_{L^2(\Omega)}=1$, by Theorem \ref{belcorped}-c), there exists $u_\delta\in \mathbb{H}_0^{\delta,s}(\Omega)$ such that $u_\delta\to u$ strong in $L^2(\Omega)$ as $\delta \to 0^+$, and
\begin{equation*}
\limsup_{\delta\to 0^+}I_{\delta,s} (u_\delta) \le I(u).
\end{equation*}
Calling $v_\delta =\frac{u_\delta}{\|u_\delta\|_{L^2(\Omega)}}$, it is elementary to check that $v_\delta\to u$ strong in $L^2(\Omega)$ as $\delta \to 0^+$ and $I_\delta^{(1)}(v_\delta)=\frac{I_\delta(u_\delta)}{\|u_\delta\|^2}$, so that
\begin{equation*}
\limsup_{\delta\to 0^+} I_{\delta,s}^{(1)}(v_\delta)=\limsup_{\delta\to 0^+} I_{\delta,s}(u_\delta)\le I^{(1)}(u).
\end{equation*}
\end{enumerate}
Since $I_{\delta,s}^{(1)}$ and $I_0^{(1)}$ are restricted functionals of $I_{{\delta},s}$ and $I_0$, compactness hold by Theorem \ref{belcorped}-a). Therefore, by Corollary \ref{corGammaconv}, we have the convergence of optimal energies,
\begin{equation*}
\kappa(N,s)\frac{\lambda_1^{\delta,s}}{\delta^{2(1-s)}}\to \lambda_1^{0,1},\quad\text{as }\delta\to 0^+,
\end{equation*}
and there exists a subsequence $\delta_n\to 0$ as $n\to\infty$, such that
\begin{equation*}
\varphi_1^{\delta_n,s}\to \varphi_1^{0,1}\quad \text{as }n\to+\infty.
\end{equation*} 


Now we prove the result for the second eigenvalue and the corresponding eigenfunction. As above, we show $\Gamma$-convergence of the restricted functionals
\begin{equation*}
     I_{\delta,s}^{(2)}(u)=\left\{\begin{array}{rl}
     I_{\delta,s}(u) &\quad\mbox{if}\quad u\in \mathbb{P}_{2}^{\delta}\text{ and }\|u\|_{L^2(\Omega)}=1,\\
                         +\infty &\quad\mbox{otherwise,}
		 \end{array}\right.
\end{equation*}
and
\begin{equation*}
     I^{(2)}(u)=\left\{\begin{array}{rl}
     I(u) &\quad\mbox{if}\quad u\in \mathbb{P}_{2}^{0}\text{ and }\|u\|_{L^2(\Omega)}=1,\\
                         +\infty &\quad\mbox{otherwise,}
		 \end{array}\right.
\end{equation*} 
where $\mathbb{P}_{2}^{\delta}$ is defined in \eqref{projectspace}, i.e.,
\begin{equation*}
\mathbb{P}_{2}^{\delta}=\left\{u\in\mathbb{H}_0^{\delta,s}(\Omega):\langle u,\varphi_1^{\delta,s}\rangle_{L^2(\Omega)}=0\right\},
\end{equation*}
and 
\begin{equation*}
\mathbb{P}_{2}^{0}=\left\{u\in H_0^1(\Omega):\langle u,\varphi_1^{0,1}\rangle_{L^2(\Omega)}=0\right\}.
\end{equation*}
$\Gamma$-convergence of $I_{\delta,s}^{(2)}$ to $I^{(2)}$ is again consequence of Lemma \ref{Glimit} and Theorem \ref{belcorped}:
\begin{enumerate}
\item {\it Liminf inequality:} Given $u_\delta\to u$ strong in $L^2(\Omega)$, with $\|u_\delta\|_{L^2(\Omega)}=1$, then $\|u\|_{L^2(\Omega)}=1$ and, hence, up to a subsequence,
\begin{equation*}
\langle u_{\delta},\varphi_1^{\delta,s}\rangle_{L^2(\Omega)}\to \langle u,\varphi_1^{0,1}\rangle_{L^2(\Omega)} \quad\text{as }\delta\to0^+,
\end{equation*}
since $\varphi_1^{\delta,s}\to \varphi_1^{0,1}$ strong in $L^2(\Omega)$, up to a subsequence. Finally, because of Theorem \ref{belcorped}-b), we find
\begin{equation}\label{liminfi}
I^{(2)}(u)\leq\liminf\limits_{\delta}I_{\delta,s}^{(2)}(u_{\delta}).
\end{equation}
\item {\it Limsup inequality:} Given $u\in \mathbb{P}_2^0$, with $\|u\|_{L^2(\Omega)}=1$, we construct a recovery sequence $u_\delta \in \mathbb{P}_2^\delta$, with $\|u_\delta\|_{L^2(\Omega)}=1$, such that $u_\delta \to u$ strong $L^2(\Omega)$ and 
\begin{equation*}
\limsup_{\delta\to 0^+} I_{\delta,s}^{(2)}(u_\delta) \le I^{(2)}(u).
\end{equation*}
Define
\begin{equation*}
u_{\delta}=\eta_{\delta}\varphi_{1}^{\delta,s}+\mu_{\delta}u,
\end{equation*}
where the constants $\eta_\delta$, $\mu_\delta$ are determined by imposing $u_{\delta}\in\mathbb{P}_{2}^{\delta}$ together with $\|u_{\delta}\|_{L^2(\Omega)}=1$. These two conditions give us
\begin{equation*}
\left\{\begin{array}{l}
     \eta_{\delta}+\mu_\delta\langle u,\varphi_1^{\delta,s}\rangle_{L^2(\Omega)}=0,\\
     \eta_{\delta}^2+2\eta_\delta\mu_\delta\langle u,\varphi_1^{\delta,s}\rangle_{L^2(\Omega)}+\mu_\delta^2=1.
		 \end{array}\right.
\end{equation*}
Since $\langle u,\varphi_1^{\delta,s}\rangle_{L^2(\Omega)}\to 0$, then $\eta_\delta\to0$ and $\mu_\delta\to1$ as $\delta\to0^+$. Then, noticing that $I_{\delta,s}$ is quadratic,
\begin{equation*}
I_{\delta,s}^{(2)}(u_\delta) =\eta_{\delta}^2 I_{\delta,s}(\varphi_{1}^{\delta,s})+2\eta_{\delta}\mu_{\delta}\langle u,\varphi_1^{\delta,s}\rangle_{L^2(\Omega)}+\mu_{\delta}^2I_{\delta,s}(u).
\end{equation*}
and taking $\delta \to 0^+$, we arrive at
\begin{equation*}
\begin{split}
\lim\limits_{\delta\to0^+} I_{\delta,s}^{(2)}(u_\delta)&= \lim\limits_{\delta\to0^+}\left(\eta_{\delta}^2 I_{\delta,s}(\varphi_{1}^{\delta,s})+2\eta_{\delta}\mu_{\delta}\langle u,\varphi_1^{\delta,s}\rangle_{L^2(\Omega)}+\mu_{\delta}^2I_{\delta,s}(u)\right)\\
&=\lim_{\delta \to 0^+} I_{\delta,s}(u)=I_0(u),
\end{split}
\end{equation*}
where the last equality is due to Theorem \ref{BoBrMi} (see Remark \ref{remarkBBM} below).

\end{enumerate}

Consequently, by Corollary \ref{corGammaconv}, 
\begin{equation*}
\kappa(N,s)\frac{\lambda_2^{\delta,s}}{\delta^{2(1-s)}}\to \lambda_2^{0,1},\quad\text{as }\delta\to0^+,
\end{equation*}
and, for any sequence $\{\varphi_2^{\delta,s}\}_{\delta>0}$ such that $\varphi_2^{\delta,s}$ is an eigenfunction of $(EP_\delta^s)$ associated to $\lambda_2^{\delta,s}$, there exists a subsequence $\{\varphi_2^{\delta_n,s}\}_{\delta_n>0}$ with $\delta_n\to 0$ as $n\to+\infty$ and $\varphi_2^{0,1}\in H_0^1(\Omega)$, an eigenfunction of $(EP_0^1)$ associated to $\lambda_2^{0,1}$, such that 
\begin{equation*}
\varphi_2^{\delta_n,s}\to \varphi_2^{0,1},\quad\text{as }n\to+\infty.
\end{equation*}
Once we have proved the convergence for the second eigenvalue and the second eigenfunction, the rest of the argument follows by induction interatively.\newline 
Therefore, we have the convergences
\begin{equation*}
\kappa(N,s)\frac{\lambda_k^{\delta,s}}{\delta^{2(1-s)}}\to\lambda_k^{0,1}\quad \text{as }\delta\to 0^+,\quad \text{for all }k\in\mathbb{N}
\end{equation*}
and, up to a subsequence,
\begin{equation}\label{cnveigenf}
\varphi_k^{\delta,s}\to\varphi_k^{0,1}\ \text{in }L^2(\Omega)\quad \text{as }\delta\to 0^+,\quad \text{for all }k\in\mathbb{N}.
\end{equation}
Moreover, because of \eqref{cnveigenf}, we also have the convergence of the orthogonal spaces
\begin{equation}\label{cnvproject}
\mathbb{P}_k^{\delta}\to\mathbb{P}_k^{0}\quad \text{as }\delta\to 0^+,\quad \text{for all }k\in\mathbb{N}.
\end{equation}
Next, given an eigenvalue $\lambda_k^{\delta,s}$, $k\in\mathbb{N}$, that, by \eqref{finitemultiplicity} has finite multiplicity, i.e., $1\leq m_k^{\delta,s}<\infty$, let us denote by $\{\varphi_{k,i}^{\delta,s}\}_{i=1}^{m_k^{\delta,s}}$ the basis of the subspace of eigenfunctions associated to the eigenvalue $\lambda_k^{\delta,s}=\ldots=\lambda_{j+m_k^{\delta,s}-1}^{\delta,s}$. Observe that, by \eqref{increasingeigenvalues}, 
\begin{equation}\label{chain}
0<\lambda_1^{\delta,s}<\lambda_2^{\delta,s}\leq\ldots\leq\lambda_{k-1}^{\delta,s}<\lambda_k^{\delta,s}=\ldots=\lambda_{k+m_k^{\delta,s}-1}^{\delta,s}<\lambda_{k+m_k^{\delta,s}}^{\delta,s}\leq\ldots
\end{equation}
Let us rewrite \eqref{chain} without repeat the eigenvalues according its multiplicity, i.e., we have a strictly increasing sequence of eigenvalues
\begin{equation}\label{chain2}
0<\lambda_1^{\delta,s}<\lambda_2^{\delta,s}<\ldots<\lambda_{k-1}^{\delta,s}<\lambda_k^{\delta,s}<\lambda_{k+1}^{\delta,s}<\ldots
\end{equation}
with finite multiplicities $\{m_k^{\delta,s}\}_{k\in\mathbb{N}}$. Because of \eqref{cnveigenf} and \eqref{cnvproject} we conclude
\begin{equation*}
m_k^{\delta,s}\to m_k^{0,1}\ \text{as }\delta\to 0^+,\quad \text{for all }k\in\mathbb{N}. 
\end{equation*}
\end{proof}
\begin{remark} \label{remarkBBM}
In order to clarify why the scaling in $\Gamma$-convergence result is natural in our context, it is of interest to obtain the upper bound
\begin{equation*}
\lim_{\delta \to 0^+} \frac{\kappa(N,s)}{\delta^{2(1-s)}}\lambda_1^{\delta,s}\le \lambda_1^{0,1}
\end{equation*} 
as a consequence of Theorem \ref{BoBrMi}. Since $H_0^1(\Omega)\subset\mathbb{H}_0^{\delta,s}(\Omega)$ for all $\delta>0$ (it is enough to extend $\varphi\in H_0^1(\Omega)$ as $\varphi\equiv0$ on $\partial_{\delta}\Omega$), then 
\begin{equation*}
\begin{split}
\lambda_1^{\delta,s}&=\min\limits_{\substack{ u\in \mathbb{H}_0^{\delta,s}(\Omega)\\ \|u\|_{L^2(\Omega)}=1}}\frac{c_{N,s}}{2}\int_{\Omega_{\delta}}\int_{\Omega_{\delta}\cap B(x,\delta)}\frac{|u(x)-u(y)|^2}{|x-y|^{N+2s}}dydx\\
&\leq \frac{c_{N,s}}{2}\int_{\Omega_{\delta}}\int_{\Omega_{\delta}\cap B(x,\delta)}\frac{|\varphi_1^{0,1}(x)-\varphi_1^{0,1}(y)|^2}{|x-y|^{N+2s}}dydx,
\end{split}
\end{equation*}
where $\varphi_1^{0,1}$ is the first eigenfunction of the Laplace operator with $\|\varphi_1^{0,1}\|_{L^2(\Omega)}=1$. In order to apply Theorem \ref{BoBrMi}, let us rewrite the above inequality as
\begin{equation*}
\begin{split}
\lambda_1^{\delta,s}&\leq \frac{c_{N,s}}{2}\int_{\Omega_{\delta}}\int_{\Omega_{\delta}\cap B(x,\delta)}\frac{|\varphi_1^{0,1}(x)-\varphi_1^{0,1}(y)|^2}{|x-y|^{N+2s}}dydx\\
&=\frac{c_{N,s}}{2}\int_{\Omega_{\delta}}\int_{\Omega_{\delta}}\frac{|\varphi_1^{0,1}(x)-\varphi_1^{0,1}(y)|^2}{|x-y|^{2}}\frac{\chi_{B(0,\delta)}(|x-y|)}{|x-y|^{N+2(s-1)}}dydx\\
&=\int_{\Omega_{\delta}}\int_{\Omega_{\delta}}\frac{|\varphi_1^{0,1}(x)-\varphi_1^{0,1}(y)|^2}{|x-y|^{2}}\rho_{\delta}(|x-y|)dydx,
\end{split}
\end{equation*}
with $\displaystyle\rho_{\delta}(z)=\frac{c_{N,s}}{2}\frac{\chi_{B(0,\delta)}(|z|)}{|z|^{N+2(s-1)}}$ and $\chi_A$ the characteristic function of the set $A$. In order to fulfill the hypotheses of Theorem \ref{BoBrMi} we normalize $\rho_{\delta}(z)$. Since
\begin{equation*}
\int \rho_{\delta}(z)dz=\frac{\sigma_{N-1}c_{N,s}}{4(1-s)}\delta^{2(1-s)},
\end{equation*}
with $\sigma_{N-1}$ the surface of the unitary sphere $\mathbb{S}^{N-1}$, the sequence of radial mollifiers
\begin{equation*}
\overline{\rho}_{\delta}(z)=\frac{4(1-s)}{\sigma_{N-1}}\frac{1}{\delta^{2(1-s)}}\frac{\chi_{B(0,\delta)}(|z|)}{|z|^{N+2(s-1)}},
\end{equation*}
satisfy the hypotheses of Theorem \ref{BoBrMi}. Observe that the scaling in $\delta$ coincides with the one of Theorem \ref{belcorped}. Therefore, we get
\begin{equation*}
\frac{4(1-s)}{\sigma_{N-1}c_{N,s}}\frac{\lambda_1^{\delta,s}}{\delta^{2(1-s)}}\leq \int_{\Omega_{\delta}}\int_{\Omega_{\delta}}\frac{|\varphi_1^{0,1}(x)-\varphi_1^{0,1}(y)|^2}{|x-y|^{2}}\overline{\rho}_{\delta}(|x-y|)dydx,
\end{equation*}
and because of Theorem \ref{BoBrMi}, we conclude
\begin{equation}\label{eigenpre}
\begin{split}
\lim\limits_{\delta\to0^+}\frac{4(1-s)}{\sigma_{N-1}c_{N,s}}\frac{\lambda_1^{\delta,s}}{\delta^{2(1-s)}}&\leq \lim\limits_{\delta\to0^+}\int_{\Omega_{\delta}}\int_{\Omega_{\delta}}\frac{|\varphi_1^{0,1}(x)-\varphi_1^{0,1}(y)|^2}{|x-y|^{2}}\overline{\rho}_{\delta}(|x-y|)dydx\\
&=\gamma\int_{\Omega}|\nabla \varphi_1^{0,1}|^2dx
\end{split}
\end{equation}
The constant $\gamma=\gamma(N,p)$ appearing in Theorem \ref{BoBrMi} takes the form (see \cite{BourBrezMiro})
\begin{equation*}
\gamma(N,p)=\frac{1}{\sigma_{N-1}}\int_{\mathbb{S}^{N-1}}|z\cdot e|^pd\sigma,
\end{equation*}
for any unitary vector $e\in\mathbb{S}^{N-1}$. Then, for $p=2$, we have $\gamma(N,2)=\frac{1}{N}$. Simplifying \eqref{eigenpre} and taking in mind that $\|\varphi_1^{0,1}\|_{L^2(\Omega)}=1$, we conclude
\begin{equation*}
\lim\limits_{\delta\to0^+}\kappa(N,s)\frac{\lambda_1^{\delta,s}}{\delta^{2(1-s)}}\leq\int_{\Omega}|\nabla \varphi_1^{0,1}|^2dx=\lambda_1^{0,1},
\end{equation*}
where $\kappa(N,s)=\frac{4N(1-s)}{\sigma_{N-1}c_{N,s}}$ is the constant appearing in \eqref{constantkappa}.\newline 

\end{remark}


\section{Taking the horizon $\delta\to+\infty$}\label{horizoninfty}

Because of the definition of the operator $(-\Delta)_{\delta}^s$, as a restriction of the fractional Laplacian, it is plausible that if we take $\delta\to+\infty$ one recovers the definition of the standard fractional Laplacian, namely
\begin{equation*}
\lim\limits_{\delta\to +\infty}(-\Delta)_{\delta}^s u(x)=c_{N,s}P.V.\int_{\mathbb{R}^N}\frac{u(x)-u(y)}{|x-y|^{N+2s}}dy.
\end{equation*}
Our interest is to explore this convergence by proving Theorems \ref{linearinf} and \ref{theigen2}. Prior to that, we would like to mention that a result in this line was given in \cite[Theorem 3.1]{DeliaGun}, where it is showed the explicit convergence rate 
\begin{equation*}
\|u^{\delta,s}-u^{\infty,s}\|_{\mathbb{H}_0^{\delta,s}}\le \frac{c}{(\delta-I)^{2s}}\|u^{\infty,s}\|_{L^2(\Omega)},
\end{equation*}
being $u^{\delta,s}$ and $u^{\infty,s}$ the solutions of $(P_\delta^s)$ and $(P_\infty^s)$ respectively, $c>0$ is a constant independent of $\delta$ and $I=\min\{R>0:\Omega\subset B(x,R)\ \forall x\in\Omega\}$. This is an important result from the point of view of the numerical approximation of problems involving the fractional Laplacian. Nevertheless, this reference does not address spectral problems, and the proof of \cite[Theorem 3.1]{DeliaGun} strongly relies on the linearity of the problem \eqref{problema}. Instead, the proof of Theorem \ref{linearinf} and Theorem \ref{theigen2} are based on a general result about $\Gamma$-convergence that works for both the linear and nonlinear setting (we address the $p\,$-fractional Laplacian case in a forthcoming paper). 
\begin{proof}[Proof of Lemma \ref{isomorfia}] 
Using \eqref{normHfrac} and \eqref{normHHfrac}, we have
\begin{equation}\label{normcomparison3}
\vertiii{v}_{\mathcal{H}_0^s}^2-\vertiii{v}_{\mathbb{H}_0^{\delta,s}}^2=\iint\limits_{\mathcal{D}\backslash\mathcal{D}_{\delta}}\frac{|v(x)-v(y)|^2}{|x-y|^{N+2s}}dydx\geq0,
\end{equation}
because $\mathcal{D}_{\delta}\subset\mathcal{D}$ for all $\delta>0$. Thus, given $v\in\mathcal{H}_0^s(\Omega)$, since $v=0$ on $\Omega^c$ we have $v=0$ on $\partial_{\delta}\Omega$ and, then, the restriction operator,
\begin{equation*}
\begin{array}{l}
     R:\mathcal{H}_0^s(\Omega)\mapsto\mathbb{H}_0^{\delta,s}(\Omega)\\
              \mkern+69.7mu v\mapsto R[v]=v\big|_{\Omega_{\delta}}         
		 \end{array}
\end{equation*}
is a continuous linear mapping. Hence, $\mathcal{H}_0^s(\Omega)$ can be continuously embedded into $\mathbb{H}_0^{\delta,s}(\Omega)$. On the other hand, for a given $v\in\mathbb{H}_0^{\delta,s}(\Omega)$, extending $v$ by 0 on $\mathbb{R}^N\backslash \Omega_{\delta}$, we have
\begin{equation*}
\begin{split}
\iint\limits_{\mathcal{D}}\frac{|v(x)-v(y)|^2}{|x-y|^{N+2s}}dydx
=&\int_{\Omega_{\delta}}\int_{\Omega_{\delta}\cap B(x,\delta)}\frac{|v(x)-v(y)|^2}{|x-y|^{N+2s}}dydx\\
&+\int_{\Omega_{\delta}}\int_{\Omega_{\delta}\backslash B(x,\delta)}\frac{|v(x)-v(y)|^2}{|x-y|^{N+2s}}dydx.
\end{split}
\end{equation*}
Using the variational characterization of the eigenvalue $\lambda_{1}^{\delta,s}$, we find
\begin{equation*}
\begin{split}
\int_{\Omega_{\delta}}\int_{\Omega_{\delta}\backslash B(x,\delta)}\frac{|v(x)-v(y)|^2}{|x-y|^{N+2s}}dydx
\leq&\frac{1}{\delta^{N+2s}}\int_{\Omega_{\delta}}\int_{\Omega_{\delta}\backslash B(x,\delta)}|v(x)-v(y)|^2dydx\\
\leq&\frac{2}{\delta^{N+2s}}\int_{\Omega_{\delta}}\int_{\Omega_{\delta}\backslash B(x,\delta)}|v(x)|^2+|v(y)|^2dydx\\
\leq&\frac{2}{\delta^{N+2s}}\int_{\Omega_{\delta}}\int_{\Omega_{\delta}}|v(x)|^2+|v(y)|^2dydx\\
=&\frac{4}{\delta^{N+2s}}\int_{\Omega_{\delta}}\int_{\Omega_{\delta}}|v(x)|^2dydx\\
=&\frac{4|\Omega_{\delta}|}{\delta^{N+2s}}\|v\|_{L^2(\Omega)}^{2}\\
\leq&\frac{4|\Omega_{\delta}|}{\delta^{N+2s}}\frac{1}{\lambda_1^{\delta,s}}\int_{\Omega_{\delta}}\int_{\Omega_{\delta}\cap B(x,\delta)}\frac{|v(x)-v(y)|^2}{|x-y|^{N+2s}}dydx.
\end{split}
\end{equation*}
Thus, $\vertiii{v}_{\mathcal{H}_0^s}^2\leq C(\delta)\vertiii{v}_{\mathbb{H}_0^{\delta,s}}^2$, where $C(\delta)=\left(1+\frac{4|\Omega_{\delta}|}{\delta^{N+2s}}\frac{1}{\lambda_1^{\delta,s}}\right)$.
As a consequence, given $v\in\mathbb{H}_0^{\delta,s}(\Omega)$, the extension operator
\begin{equation*}
\begin{array}{l}
     E:\mathbb{H}_0^{\delta,s}(\Omega)\mapsto\mathcal{H}_0^s(\Omega)\\
              \mkern+79.2mu v\mapsto E[v]= \left\{\begin{array}{rl}
                                                   v&\mbox{in}\quad \Omega_{\delta},\\
                                                   0&\mbox{in}\quad \mathbb{R}^N\backslash\Omega_{\delta},\\
		                                              \end{array}\right.       																							\end{array}
\end{equation*}
is a linear continuous mapping so that $\mathbb{H}_0^{\delta,s}(\Omega)$ can be continuously embedded into $\mathcal{H}_0^s(\Omega)$. Now, by \eqref{contenido} and \eqref{normcomparison3}, for any positive $\delta_1,\,\delta_2$ with $\delta_1<\delta_2$,
\begin{equation*}
\mathbb{H}_0^{\delta_1,s}(\Omega)\subset\mathbb{H}_0^{\delta_2,s}(\Omega)\subset \mathcal{H}_0^s(\Omega).
\end{equation*}
In particular, the sequence of eigenvalues $\{\lambda_1^{\delta,s}\}_{\delta>0}$ is increasing in $\delta$ and uniformly bounded from above by the first eigenvalue $\lambda_{1}^{\infty,s}$ of the fractional Laplacian $(-\Delta)_{\infty}^{s}$. Therefore, the constant $C(\delta)$ satisfies
\begin{equation*}
\begin{split}
C(\delta)&=1+\frac{4|\Omega_{\delta}|}{\delta^{N+2s}}\frac{1}{\lambda_1^{\delta,s}}\leq1+c\frac{(diam(\Omega)+2\delta)^N}{\delta^{N+2s}}\frac{1}{\lambda_1^{\delta,s}}\to 1,\quad\text{as }\delta\to+\infty.
\end{split}
\end{equation*} 
\end{proof}
The following $\Gamma$-convergence result is in the core of the proofs of Theorems \ref{linearinf} and Theorem \ref{theigen2}.
\begin{lemma}\label{Gconvergenceinf}Let us consider the functional 
\begin{equation*}
\mathcal{E}_{\delta,s}(u)=\int_{\Omega_{\delta}}\int_{\Omega_{\delta}\cap B(x,\delta)}\frac{|u(x)-u(y)|^2}{|x-y|^{N+2s}}dydx.
\end{equation*} 
defined on $\mathbb{H}_0^{\delta,s}(\Omega)$. Then, the $\Gamma$-limit of $\mathcal{E}_{\delta,s}(u)$ is given by
\begin{equation}\label{GlimitFuncinf}
\mathcal{E}_{\infty,s}(u)=\int_{\mathbb{\mathbb{R}}^N}\int_{\mathbb{R}^N}\frac{|u(x)-u(y)|^2}{|x-y|^{N+2s}}dydx\quad\text{as }\delta\to+\infty.
\end{equation}
\end{lemma}
\begin{proof}
The sequence of functionals $\mathcal{E}_{\delta,s}(u)$ with $\delta\to+\infty$ is a monotone increasing sequence and functionals $\mathcal{E}_{\delta,s}$ are lower semicontinuous, cf. \cite{BellCor}. Therefore, because of \cite[Remark 1.40]{Braides}, 
\begin{equation*}
\mathcal{E}_{\delta,s}(u)\overset{\Gamma}{\to}\mathcal{E}_{\infty,s}(u)=\int_{\mathbb{\mathbb{R}}^N}\int_{\mathbb{R}^N}\frac{|u(x)-u(y)|^2}{|x-y|^{N+2s}}dydx.
\end{equation*}
\end{proof}
\begin{proof}[Proof of Theorem \ref{linearinf}]
Let us consider the energy functional associated to problem \eqref{problema},
\begin{equation*}
\begin{split}
\mathcal{J}_{\delta,s}(u)&=\frac{c_{N,s}}{4}\int_{\Omega_{\delta}}\int_{\Omega_{\delta}\cap B(x,\delta)}\frac{|u(x)-u(y)|^2}{|x-y|^{N+2s}}dydx-\int_{\Omega}fudx\\
&=\frac{c_{N,s}}{4}\mathcal{E}_{\delta,s}(u)-\int_{\Omega}fudx.
\end{split}
\end{equation*}
Then, by Lemma \ref{Gconvergenceinf}, we conclude
\begin{equation*}
\mathcal{J}_{\delta,s}(u)\overset{\Gamma}{\to}\mathcal{J}_{\infty,s}(u)=\frac{c_{N,s}}{4}\int_{\mathbb{\mathbb{R}}^N}\int_{\mathbb{R}^N}\frac{|u(x)-u(y)|^2}{|x-y|^{N+2s}}dydx-\int_{\Omega}fudx,
\end{equation*}
that is the energy functional associated to \eqref{problemainf}.\newline 
Now, if $u^{\delta,s}$ is the minimizer of $\mathcal{J}_{\delta,s}$, the sequence $\{\mathcal{J}_{\delta,s}(u^{\delta,s})\}_{\delta>0}$ is monotone increasing in $\delta$ and bounded from above by $\mathcal{J}_{\infty,s}(u^{\infty,s})$. Indeed, given $\delta_1<\delta_2$, then 
\begin{equation*} 
\mathcal{J}_{\delta_1,s}(u^{\delta_1,s})\le \mathcal{J}_{\delta_1,s}(u^{\delta_2,s}) \le \mathcal{J}_{\delta_2,s}(u^{\delta_2,s}).
\end{equation*} 
On the other hand, since
\begin{equation*}
\mathcal{J}_{\delta,s}(u^{\delta,s})=-\frac{c_{N,s}}{2} \vertiiin{u^{\delta,s}}_{\mathbb{H}_0^{\delta,s}}^2,
\end{equation*}
the sequence $\vertiiin{u^{\delta,s}}_{\mathbb{H}_0^{\delta,s}}$ is decreasing and bounded. Consequently, by Lemma \ref{isomorfia}, $\vertiiin{u^{\delta,s}}_{\mathcal{H}_0^{s}}$ is bounded and, by Lemma \ref{isomorfia} and the compact embedding of $\mathcal{H}_0^s(\Omega)$ into $L^2(\Omega)$, cf. \cite[Corollary 7.2]{DiNezzaPalaValdi}, there exists a subsequence (that we do not relabel) and $u^{\infty,s}\in \mathcal{H}_0^s(\Omega)$ such that
\begin{equation*} 
u^{\delta,s}\to u^{\infty,s}\quad\text{as }\delta\to+\infty.
\end{equation*}
The $\Gamma$-convergence of $\mathcal{J}_{\delta,s}$ to $\mathcal{J}_{\infty,s}$ implies that $u^{\infty,s}$ is the unique minimizer of $\mathcal{J}_{\infty,s}$.
\end{proof}

\begin{proof}[Proof of Theorem \ref{theigen2}]
The proof follows by combining Lemma \ref{Gconvergenceinf} and Lemma $\ref{isomorfia}$ with the arguments used in the proof of Theorem \ref{theigen}. 

\end{proof}



\begin{thebibliography}{10}

\bibitem{Alali}
B. Alali and N. Albin, \textit{Fourier multipliers for nonlocal Laplace
  operators}, Applicable Analysis, {\bfseries 0} (2019), pp. 1--21.

\bibitem{AlBe}
G. Alberti and G. Bellettini, \textit{A non-local anisotropic model for phase
  transitions: asymptotic behaviour of rescaled energies}, European J. Appl.
  Math., {\bfseries 9} (1998), pp. 261--284.

\bibitem{AndresMunoz}
F. Andr\'{e}s and J. Mu\~{n}oz, \textit{A type of nonlocal elliptic problem:
  existence and approximation through a Galerkin-Fourier method}, SIAM J.
  Math. Anal., {\bfseries 47} (2015), pp. 498--525.

\bibitem{Mazon}
F. Andreu-Vaillo, J. M. Maz\'{o}n, J. D. Rossi  and J. J. Toledo-Melero, 
\textit{Nonlocal diffusion problems}, vol. {\bfseries 165} of Mathematical Surveys
  and Monographs, American Mathematical Society, Providence, RI; Real Sociedad
  Matem\'{a}tica Espa\~{n}ola, Madrid, 2010.

\bibitem{BellCor}
J. C. Bellido and C. Mora-Corral, \textit{Existence for nonlocal variational
  problems in peridynamics}, SIAM J. Math. Anal., {\bfseries 46} (2014),
  pp. 890--916.

\bibitem{BellCorPed}
J. C. Bellido, C. Mora-Corral and P. Pedregal, \textit{Hyperelasticity as a
  {$\Gamma$}-limit of peridynamics when the horizon goes to zero}, Calc. Var.
  Partial Differential Equations, {\bfseries 54} (2015), pp. 1643--1670.

\bibitem{BoElPonScher}
J. Boulanger, P. Elbau, C. Pontow  and O. Scherzer, \textit{Non-local functionals
  for imaging}, in Fixed-point algorithms for inverse problems in science and
  engineering, vol. {\bfseries 49} of Springer Optim. Appl., Springer, New
  York, 2011, pp. 131--154.

\bibitem{BourBrezMiro}
J. Bourgain, H. Brezis and P. Mironescu, \textit{Another look at Sobolev
  spaces}, in Optimal control and partial differential equations, IOS,
  Amsterdam, 2001, pp. 439--455.

\bibitem{Braides}
A. Braides, \textit{{$\Gamma$}-convergence for beginners}, vol. {\bfseries 22} of
  Oxford Lecture Series in Mathematics and its Applications, Oxford University
  Press, Oxford, 2002.

\bibitem{BrascoPariniSquassina}
L. Brasco, E. Parini and M. Squassina, \textit{Stability of variational
  eigenvalues for the fractional {$p$}-Laplacian}, Discrete Contin. Dyn.
  Syst., {\bfseries 36} (2016), pp. 1813--1845.

\bibitem{Brezis}
H. Brezis, \textit{Functional analysis, Sobolev spaces and partial differential
  equations}, Universitext, Springer, New York, 2011.

\bibitem{Bucur}
C. Bucur, \textit{Some observations on the Green function for the ball in the
  fractional Laplace framework}, Commun. Pure Appl. Anal., {\bfseries 15} (2016),
  pp. 657--699.

\bibitem{BuVal}
C. Bucur and E. Valdinoci, \textit{Nonlocal diffusion and applications},
  vol. {\bfseries 20} of Lecture Notes of the Unione Matematica Italiana,
  Springer, [Cham]; Unione Matematica Italiana, Bologna, 2016.

\bibitem{DeliaGun}
M. D'Elia and M. Gunzburger, \textit{The fractional Laplacian operator on
  bounded domains as a special case of the nonlocal diffusion operator},
  Comput. Math. Appl., {\bfseries 66} (2013), pp. 1245--1260.

\bibitem{DiNezzaPalaValdi}
E. Di Nezza, G. Palatucci and E. Valdinoci, \textit{Hitchhiker's guide to the
  fractional Sobolev spaces}, Bull. Sci. Math., {\bfseries 136} (2012),
  pp. 521--573.

\bibitem{Du}
Q. Du, \textit{Nonlocal modeling, analysis, and computation}, vol. {\bfseries 94}
  of CBMS-NSF Regional Conference Series in Applied Mathematics, Society for
  Industrial and Applied Mathematics (SIAM), Philadelphia, PA, 2019.

\bibitem{Siwei}
S. Duo, H. Wang and Y. Zhang, \textit{A comparative study on nonlocal diffusion
  operators related to the fractional Laplacian}, Discrete Contin. Dyn. Syst.
  Ser. B, {\bfseries 24} (2019), pp. 231--256.

\bibitem{EvBe}
A. Evgrafov and J. C. Bellido, \textit{From non-local Eringen's model to
  fractional elasticity}, Math. Mech. Solids, {\bfseries 24} (2019), pp. 1935--1953.

\bibitem{GilStan}
G. Gilboa and S. Osher, \textit{Nonlocal operators with applications to image
  processing}, Multiscale Model. Simul., {\bfseries 7} (2008), pp. 1005--1028.

\bibitem{KassMenScott}
M. Kassmann, T. Mengesha and J. Scott, \textit{Solvability of nonlocal systems
  related to peridynamics}, Commun. Pure Appl. Anal., {\bfseries 18} (2019),
  pp. 1303--1332.

\bibitem{KinOJo}
S. Kindermann, S. Osher and P. W. Jones, \textit{Deblurring and denoising of
  images by nonlocal functionals}, Multiscale Model. Simul., {\bfseries 4} (2005),
  pp. 1091--1115.

\bibitem{GioSpec}
G. Leoni and D. Spector, \textit{Characterization of Sobolev and {$BV$} spaces},
  J. Funct. Anal., {\bfseries 261} (2011), pp. 2926--2958.

\bibitem{LionMag}
J.-L. Lions and E. Magenes, \textit{Non-homogeneous boundary value problems and
  applications. Vol. I}, Springer-Verlag, New York-Heidelberg, 1972.
\newblock Translated from the French by P. Kenneth, Die Grundlehren der
  mathematischen Wissenschaften, Band 181.

\bibitem{MR4043885}
A. Lischke, G. Pang, M. Gulian and et al., \textit{What is the fractional
  Laplacian? A comparative review with new results}, J. Comput. Phys., {\bfseries
  404} (2020), 109009, 62 pp.

\bibitem{Men}
T. Mengesha, \textit{Nonlocal Korn-type characterization of Sobolev vector
  fields}, Commun. Contemp. Math., {\bfseries 14} (2012), 1250028, 28 pp.

\bibitem{MenDu}
T. Mengesha and Q. Du, \textit{Characterization of function spaces of vector
  fields and an application in nonlinear peridynamics}, Nonlinear Anal., {\bfseries
  140} (2016), pp. 82--111.

\bibitem{MenSpec}
T. Mengesha and D. Spector, \textit{Localization of nonlocal gradients in various
  topologies}, Calc. Var. Partial Differential Equations, {\bfseries 52} (2015),
  pp. 253--279.

\bibitem{Ponce}
A. C. Ponce, \textit{A new approach to Sobolev spaces and connections to
  {$\Gamma$}-convergence}, Calc. Var. Partial Differential Equations, {\bfseries 19}
  (2004), pp. 229--255.

\bibitem{RosOton}
X. Ros-Oton, \textit{Nonlocal elliptic equations in bounded domains: a survey},
  Publ. Mat., {\bfseries 60} (2016), pp. 3--26.

\bibitem{RosOton2012}
X. Ros-Oton and J. Serra, \textit{Fractional Laplacian: Pohozaev identity and
  nonexistence results}, Comptes Rendus Mathematique, {\bfseries 350} (2012),
  pp. 505--508.

\bibitem{Servadei2013}
R. Servadei and E. Valdinoci, \textit{A Brezis-Nirenberg result for non-local
  critical equations in low dimension}, Communications on Pure {\&} Applied
  Analysis, {\bfseries 12} (2013), pp. 2445--2464.

\bibitem{SerVal2}
R. Servadei and E. Valdinoci, \textit{Mountain pass solutions for non-local
  elliptic operators}, J. Math. Anal. Appl., {\bfseries 389} (2012), pp. 887--898.

\bibitem{SerVal}
R. Servadei and E. Valdinoci, \textit{Variational methods for non-local operators
  of elliptic type}, Discrete Contin. Dyn. Syst., {\bfseries 33} (2013),
  pp. 2105--2137.

\bibitem{Servadei2014}
R. Servadei and E. Valdinoci, \textit{On the spectrum of two different fractional
  operators}, Proceedings of the Royal Society of Edinburgh: Section A
  Mathematics, {\bfseries 144} (2014), pp. 831--855.

\bibitem{Si}
S. A. Silling, \textit{Reformulation of elasticity theory for discontinuities and
  long-range forces}, J. Mech. Phys. Solids, {\bfseries 48} (2000), pp. 175--209.

\bibitem{Spec}
D. Spector, \textit{On a generalization of {$L^p$}-differentiability}, Calc. Var.
  Partial Differential Equations, {\bfseries 55} (2016), Art. 62, 21 pp.

\bibitem{Stinga}
P. R. Stinga and J. L. Torrea, \textit{Extension problem and Harnack's
  inequality for some fractional operators}, Comm. Partial Differential
  Equations, {\bfseries 35} (2010), pp. 2092--2122.

\end{thebibliography}

\end{document}